\documentclass[final]{siamart171218}
\usepackage{graphicx}
\usepackage{hyperref}
\usepackage{amsmath}
\usepackage{amssymb}
\usepackage[numbers]{natbib}
\usepackage{xcolor}
\usepackage{enumerate}
\usepackage{cleveref}

\newsiamthm{example}{Example}
\newsiamthm{remark}{Remark}

\newcommand{\complex}{\mathbb{C}}
\newcommand{\real}{\mathbb{R}}
\newcommand{\nat}{\mathbb{N}}
\newcommand{\identity}{\mathbb{I}}

\newcommand{\nullspace}{\mathcal{N}}
\newcommand{\range}{\mathcal{R}}
\renewcommand{\Re}{\text{Re}\,}
\renewcommand{\Im}{\text{Im}\,}

\newcommand{\M}[1]{\left({#1}\right)}

\newcommand{\Mcb}[1]{\left\{{#1}\right\}}

\newcommand{\kron}{\otimes}
\newcommand{\hadamard}{\circledast}
\newcommand{\outter}{\odot}

\newcommand{\proba}{\mathbb{P}}

\newcommand{\cF}{\mathcal{F}}

\DeclareMathOperator{\diag}{diag}
\DeclareMathOperator{\rank}{rank}
\DeclareMathOperator{\linspan}{span}
\DeclareMathOperator{\vect}{vec}

\DeclareMathOperator{\sign}{sign}

\title{Matrix valued inverse problems on graphs\\with application to
elastodynamic networks\thanks{\funding{This work was supported by the
National Science Foundation grants DMS-1411577 and DMS-1439786.}}}
\author{F. Guevara Vasquez\thanks{Mathematics Dept., U. of Utah, 155 S
1400 E RM 233, 84112-0090, \email{fguevara@math.utah.edu}.}
\and T. G. Draper\footnotemark[2]%
\and J. C.-L. Tse\footnotemark[2]%
\and T. E. Wallengren\footnotemark[2]%
\and K. Zheng\footnotemark[2]%
}
\headers{F.~Guevara~Vasquez et al.}{Matrix valued inverse problems}

\crefname{subsection}{section}{sections}
\Crefname{subsection}{Section}{Sections}
\crefname{equation}{}{}

\begin{document}
\maketitle
\begin{abstract} 
We consider the inverse problem of finding matrix valued edge or nodal
quantities in a graph from measurements made at a few boundary nodes.
This is a generalization of the problem of finding resistors in a
resistor network from voltage and current measurements at a few nodes,
but where the voltages and currents are vector valued.
The measurements come from solving a series of Dirichlet problems, i.e.
finding vector valued voltages at some interior nodes from voltages
prescribed at the boundary nodes. We give conditions under which the
Dirichlet problem admits a unique solution and study the degenerate case
where the edge weights are rank deficient. Under mild conditions, the
map that associates the matrix valued parameters to boundary data is
analytic.  This has practical consequences to iterative methods for
solving the inverse problem numerically and to local uniqueness of the
inverse problem. Our results allow for complex valued weights and give
also explicit formulas for the Jacobian of the parameter to data map in
terms of certain products of Dirichlet problem solutions. An application
to inverse problems arising in elastodynamic networks (networks of
springs, masses and dampers) is presented.
\end{abstract}



\section{Introduction}
\label{sec:intro}

We study a class of inverse problems where the objective is to find
matrix valued quantities defined on the edges or vertices (nodes) of a graph
from measurements made at a few {\em boundary nodes}.  The scalar
case corresponds to the problem of finding resistors in a resistor
network from electrical measurements made at a few nodes, see e.g.
\cite{Curtis:1998:CPG}. As in the scalar case, the vector potential at
all the nodes can be found from its value at a few nodes by solving a
Dirichlet problem which amounts to finding a vector potential satisfying
a vector version of conservation of currents (Kirchhoff's node law).

We present different inverse problems, where either the matrix valued
weights on the edges or the vertices or even their eigenvalues are the
unknown parameters that are sought after. All these inverse problems
share a common structure that is given in \cref{sec:structure}. Any inverse problem that
fits this mold has certain desirable properties: mainly the parameter to
data map (i.e. the forward map) is analytic and its Jacobian can be
computed in terms of products of internal states. Analyticity can be
used to guarantee local uniqueness for such inverse problems, for almost
any parameter within a region of interest provided the Jacobian is
injective for one parameter (a generalization of the results in
\cite{Boyer:2015:SDC}). Moreover, we show that Newton's method applied
to such problems is very likely to produce valid steps.  We study in
detail the Dirichlet boundary value problem on graphs with matrix valued
weights and give conditions under which the Dirichlet problem admits a
unique solution (\cref{sec:matrix:problems,sec:diruni}). Our study
includes cases where the matrix valued weights are rank deficient and
uniqueness holds only up to a known subspace. Then in
\cref{sec:matrixval:ex,sec:elnets} we formulate inverse problems with
matrix valued weights and determine conditions under which they have the
structure of \cref{sec:structure}.  Some of the inverse problems we
consider arise in elastodynamic networks, i.e. networks of springs,
masses and dampers. 

\subsection{Related work}

The discrete conductivity inverse problem consists in finding the
resistors in a resistor network from voltage and current measurements
made at a few nodes, assuming the underlying graph is known. For his
problem, the uniqueness results in
\cite{Curtis:1994:FCC,Curtis:1998:CPG,Colin:1994:REP,Colin:1996:REP,Colin:1998:SG}
apply to circular planar graphs and real conductivities. A different
approach is taken in \cite{Chung:2010:IRE} where a monotonicity property
inspired from the continuum \cite{Alessandrini:1989:RPB} is used to show
that if the conductivities satisfy a certain inequality then they can be
uniquely determined from measurements, without specific assumptions on
the underlying graph. The lack of uniqueness is shown for cylindrical
graphs in \cite{Lam:2012:IPC}. For complex conductivities, a condition
for ``uniqueness almost everywhere'' regardless of the underlying graph
is given in \cite{Boyer:2015:SDC}. Uniqueness almost everywhere means
that the set of conductivities that have the same boundary data lie in a
zero measure set and that the linearized problem is injective for almost
all conductivities in some region.

Uniqueness for the discrete Schr\"odinger problem is considered in the
real scalar case on circular planar graphs in
\cite{Arauz:2014:DRM,Arauz:2015:DRM,Arauz:2015:OPB}. This problem
involves a resistor network with known underlying graph and resistors
but where every node is connected to the ground (zero voltage) via a
resistor with unknown resistance. These unknown resistors are a discrete
version of the Schr\"odinger potential in the Schr\"odinger equation,
and the goal is to find them from measurements made at a few nodes. A
discrete Liouville identity \cite{Borcea:2017:DLI} can be used to relate
the discrete Schr\"odinger inverse problem for certain Schr\"odinger
potentials to the discrete conductivity inverse problem, also on
circular planar graphs.  A condition guaranteeing uniqueness almost
everywhere for complex valued potentials without an assumption on the
graph is given in \cite{Boyer:2015:SDC}.

One of the consequences of the present study is a weak uniqueness result
for matrix valued inverse problems on graphs.  To the best of our
knowledge there are no results for uniqueness of the inverse problem
with matrix valued edge or node quantities other than the
characterization and synthesis results for elastodynamic networks
(discussed in more detail in \cref{sec:elnets}) that are derived in
\cite{Camar:2003:DCS,Guevara:2011:CCS,Gondolo:2014:CSR}. These results
solve an inverse problem for elastodynamic networks that assumes we are
free to choose the graph topology. Indeed the constructions in
\cite{Camar:2003:DCS,Guevara:2011:CCS,Gondolo:2014:CSR} start from data
generated by these networks (displacement to forces map) and give a
network that reproduces this data.  We emphasize that in the present
study, the underlying graph is always assumed to be known.

\section{Common structure}
\label{sec:structure}
The discrete inverse problems we consider here share a common structure
that we describe in \cref{sec:abstract}. Under the assumptions we make here,
the linearization of the problem is readily available (\cref{sec:jac})
and analyticity of the forward map is ensured. This has practical
implications that are described in \cref{sec:uniqueness:ae}.

\subsection{An abstract inverse problem}
\label{sec:abstract}
We
denote by $p \in \complex^m$ the unknown {\bf parameter}. As we see later in
\cref{sec:matrixval:ex,sec:elnets}, the parameter $p$ may represent a
matrix valued quantity (or its eigenvalues) defined on the edges or
nodes of a graph. The forward or parameter to data map associates to the parameter $p$ the
matrix $\Lambda_p \in \complex^{n \times n}$ (the {\bf data}), provided
the parameter $p$ belongs to an admissible set $R \subset \complex^m$ of
parameters. The inverse problem is to find $p$ from $\Lambda_p$.
Furthermore, we assume that the discrete inverse problems we consider
satisfy the following assumptions.
\begin{itemize}
 \item {\bf Assumption 1.} The
  parameter $p$ belongs to an open convex set $R \subset \complex^m$ of
  {\bf admissible parameters}. The {\bf forward map} that to a parameter
  $p \in R$ associates the data $\Lambda_p$ is well defined for $p \in
  R$.

 \item {\bf Assumption 2.} For all $f,g \in
  \complex^n$ and $p_1,p_2 \in R$ the following {\bf boundary/interior
  identity} holds:
 \begin{equation}
 \label{eq:intid}
  f^T (\Lambda_{p_1} - \Lambda_{p_2}) g = b(S_{p_2} g,S_{p_1} f)^T  (p_1
  - p_2),
 \end{equation}
  where $b: \complex^\ell \times \complex^\ell \to \complex^m$ is a
  bilinear mapping and  $S_p \in \complex^{\ell \times n}$ is a matrix
  defined for $p \in R$ that associates to a boundary
  condition $f \in \complex^n$, an internal ``state'' $S_p f \in
  \complex^\ell$.  

 \item {\bf Assumption 3: Analyticity.} The entries of $S_p$ are analytic
 functions of $p$ for $p\in R$. 
\end{itemize}
Here by ``analytic'' we mean in the sense of analyticity of
several complex variables, see e.g.  \cite{Gunning:1965:AFS}. For
completeness, we recall in \cref{app:faf} all the results we use from
the theory of functions of several complex variables. We note that the
boundary/interior identity \eqref{eq:intid} is a discrete version of a
similar identity that plays a key role in the Sylvester and Uhlmann
\cite{Sylvester:1987:GUT} proof of uniqueness for the continuum
Schr\"odinger inverse problem.

\subsection{The product of solutions matrix and the Jacobian}
\label{sec:jac}
For a discrete inverse problem satisfying
assumptions 1--3, we define the following {\em product of
solutions matrix}, which is the matrix valued function $W: R \times R
\to \complex^{m \times n^2}$ with columns given by
\begin{equation}
  [W (p_1,p_2)](:,i+(j-1)n) = 
  b(S_{p_1}(:,i),S_{p_2}(:,j)),
  ~i,j=1,\ldots,n.
  \label{eq:wmatrix}
\end{equation}
The next lemma shows that the parameter to data map $\Lambda_p$ must be
Fr\'echet differentiable (specialized versions of this lemma appear in 
\cite[lemma 5.4  and 6.3]{Boyer:2015:SDC}).
\begin{lemma}[Linearization of discrete inverse problem]
\label{lem:lin}
 Let $p \in R$. For sufficiently small $\delta p \in  \complex^m$, we have
 \begin{equation}
  f^T\Lambda_{p+\delta p} g = f^T \Lambda_p g + 
  b(S_p f,S_p g)^T\delta p  + o(\delta p).
 \end{equation}
\end{lemma}
\begin{proof}
Use the boundary/interior identity \eqref{eq:intid} with $p_1 = p + \epsilon \delta
p$ and $p_2 = p$, for some scalar $\epsilon$. To conclude divide both
sides by $\epsilon$ and take the limit as $\epsilon \to 0$. Notice that
assumption 3 guarantees that
$S_p$ is analytic in $p$, therefore we do have continuity of $S_p$ in
$p$ and $S_{p + \epsilon \delta p} \to S_p$ as $\epsilon \to 0$.
\end{proof}

A consequence of \cref{lem:lin} is that $W(p,p)^T$ is a $n^2 \times m$
matrix representation of the Jacobian matrix for the parameter to data
map at parameter value $p$. From \eqref{eq:wmatrix}, the representation is associated to
identifying the matrix $\Lambda_p \in \complex^{n \times n}$  with the
vector $\vect(\Lambda_p) \in \complex^{n^2}$, which is obtained by
stacking the columns of $\Lambda_p$.  Clearly the linearized
inverse problem about $p$ is {\em injective} when $\nullspace(W(p,p)^T)
= \{0\}$, i.e. when the product of solutions matrix $W(p,p)$ has full
row rank, i.e. $\range (W(p,p)) = \complex^m$.

Another consequence of \cref{lem:lin} is that the Jacobian of
$\Lambda_p$ with respect to $p$ must be analytic for $p \in R$ (by
assumption 3). Clearly the forward map $\Lambda_p$ must also be analytic
for $p \in R$.

\subsection{Analyticity and uniqueness almost everywhere}
\label{sec:uniqueness:ae}

We look at the impact of analyticity on the uniqueness question: 
\begin{quote}
 If $p_1, p_2 \in \complex^m$ are parameters with
identical data $\Lambda_{p_1} = \Lambda_{p_2}$, can we conclude that $p_1 =
p_2$?
\end{quote}
For inverse problems satisfying assumptions 1--3, we can only guarantee
uniqueness in a weak sense that we call {\em uniqueness almost
everywhere} (as in \cite{Boyer:2015:SDC}). By this we mean that the
linearized problem is injective for almost all parameters $p \in R$ and
that the set of parameters having the data must be a zero measure set.
Both properties follow readily from analyticity, as we see next.

Analyticity of the forward map $\Lambda_p$ readily gives uniqueness
almost everywhere, meaning that the sets of parameters that have the
same boundary data must be of zero measure. Indeed assume we can find
$\rho_1,\rho_2 \in R^2$ such that $\Lambda_{\rho_1} \neq
\Lambda_{\rho_2}$. Then we can consider the function $g: R \times R \to
\complex$ defined by $g(x,y) = [\Lambda_{x} - \Lambda_y]_{ij}$ for some
$i,j\in 1,\ldots,n$. Clearly $g$ is analytic on $R^2$ and satisfies
$g(\rho_1,\rho_2) \neq 0$. By analytic continuation, the set $\{
(p_1,p_2) \in R \times R ~|~ g(p_1,p_2) = 0\}$ must be a zero measure
set. This is a much simpler way of reaching a result similar to in
\cite{Boyer:2015:SDC} and was suggested by Druskin
\cite{Druskin:2015:PC}.

Analyticity can also be used to deduce that if the Jacobian of the
forward map is injective at a parameter $\rho \in R$, then it must be
invertible at almost any other parameter $p \in R$. Indeed,
\cref{lem:lin} shows that the Jacobian at $p$ can be represented by
the $n^2 \times m$ matrix $W(p,p)^T$ defined in \eqref{eq:wmatrix}.
If $W(\rho,\rho)^T$ is injective for a $\rho \in R$, then there is a $m \times m$ submatrix
$[W(\rho,\rho)]_{:,\alpha}$ of $W(\rho,\rho)$ that is invertible, where
$\alpha = (\alpha_1,\ldots,\alpha_m) \in \{1,\ldots,n^2\}^m$ is a
multi-index used to represent the particular choice of columns.
Thus the function $f: R \to \complex$ defined by
\begin{equation}
 f(p) = \det [W(p,p)]_{:,\alpha}
 \label{eq:detfun}
\end{equation} 
is analytic for $p \in R$ and is such that $f(\rho) \neq 0$. By analytic
continuation, the zero set of $f$ must be of measure zero. Thus the set
of parameters for which the Jacobian is not injective must be a zero
measure set.

Finally we note that if we can find a parameter $\rho$ for which the
Jacobian $W(\rho,\rho)^T$ is injective, then we can use the constant
rank theorem (see e.g. \cite{Rudin:1976:PMA}) to show that there is a
$\rho' \in R$ in a neighborhood of $\rho$ such that $\Lambda_\rho \neq
\Lambda_{\rho'}$. Therefore the set of parameters that have the same
data must be a zero measure set.

\subsection{Applications of uniqueness almost everywhere}
\label{sec:applications}
Uniqueness a.e. has several
practical applications that are illustrated for the scalar discrete
conductivity problem in \cite{Boyer:2015:SDC}. We give an outline of
these applications for completeness. The first application is a simple
test to determine whether uniqueness a.e. holds for a
particular discrete inverse problem and that may also indicate
sensitivity to noise (\cref{sec:uniqueness:test}). Once we know
uniqueness a.e. holds for a particular discrete inverse
problem, we can guarantee that the situations in which Newton's method
with line search fails can be easily avoided (\cref{sec:newton}).
Naturally a statement about zero measure sets can be
translated to a probabilistic setting (\cref{sec:proba}).

\subsubsection{A test for uniqueness almost everywhere}
\label{sec:uniqueness:test}
Recall from \cref{lem:lin} that the
Jacobian of the discrete inverse problem at a parameter $p$ can be
easily computed as a products of solutions matrix \eqref{eq:wmatrix}
with $p \equiv p_1 = p_2$. As discussed in \cref{sec:uniqueness:ae}, if we
can find a parameter $p \in R$ for which the Jacobian is injective, then
uniqueness a.e. holds for the problem. A numerical test for
uniqueness a.e. can be summarized as follows.

\begin{enumerate}
 \item Pick a parameter $p \in R$.
 \item Calculate the Jacobian $W(p,p)^T$ using \eqref{eq:wmatrix}.
 \item Find the largest and smallest singular values $\sigma_{\max},\sigma_{\min}$ of $W(p,p)$.
 \item If $\sigma_{\min} > \epsilon \sigma_{\max}$, where $\epsilon$ is a tolerance
 set a priori, then uniqueness a.e. holds. 
\end{enumerate}

We point out that if $\sigma_{\min} \leq \epsilon \sigma_{\max}$ it is not
possible to distinguish between the two following scenarios: (a) uniqueness a.e. holds but $W(p,p)$
is not injective to precision $\epsilon$; or (b) uniqueness a.e. does not hold
for the problem.  Thus the test is inconclusive. However we know that scenario
(a) is very unlikely because we would have had to pick $p$ on the zero measure
subset of $R$ that contains all parameters for which the Jacobian is not
injective. Thus the most likely outcome is (b).  Finally we remark that other
methods may be used instead of the Singular Value Decomposition (SVD) to find
the rank of the Jacobian (e.g. the QR factorization). We prefer the SVD because
the ratio $\sigma_{\max}/\sigma_{min}$ is the conditioning of the linear least
squares problem associated with the linearization of the discrete inverse
problem, and thus measures the sensitivity to noise of the linearization of the inverse
problem about the parameter $p$.

\subsubsection{Newton's method}
\label{sec:newton}
The discrete inverse problem of finding the parameter $p$ from the
data $\Lambda_p$ is a non-linear system of equations that can be solved
using Newton's method (see e.g. \cite{Nocedal:2006:NO}). Let us denote by $D
\Lambda_p = W(p,p)^T$ the Jacobian of the Dirichlet to Neumann map about the
parameter $p \in R$. For our particular problem
we get the following.
\begin{tabbing}
xxx \= xxx \= xxx\kill\\ {\bf Newton's method}\\
$p^{(0)} =$ given\\
for $k=0,1,2,\ldots$\\
\> Find step $\delta p^{(k)}$ s.t. $D\Lambda_{p^{(k)}} \delta
p^{(k)} = \vect(\Lambda_{p^{(k)}} - \Lambda_{p})$\\
\> Choose step length $t_k>0$\\
\> Update $p^{(k+1)} = p^{(k)} + t_k \delta p^{(k)}$\\
\end{tabbing}
The first operation in the Newton iteration is to solve a linear problem
for the step $\delta p^{(k)}$. This operation can fail either because $
\vect(\Lambda_{p^{(k)}} - \Lambda_{p}) \notin
\range(D\Lambda_{p^{(k)}})$ or because
$\nullspace(D\Lambda_{p^{(k)}}) \neq \{ 0 \}$. A remedy to either of
these situations is to solve the linear least squares system
\begin{equation}
 \label{eq:step}
 \min_{\delta p} \| D\Lambda_{p^{(k)}} \delta p - \vect(\Lambda_{p^{(k)}} -
 \Lambda_{p}) \|_2^2,
\end{equation}
and pick $\delta p^{(k)}$ as the minimal norm solution to
\eqref{eq:step}.  If uniqueness a.e. holds for the problem at hand then
clearly $D\Lambda_{p^{(k)}}$ is injective except on a zero measure
set.  Therefore we can expect the step in Newton's method to be defined
uniquely. Now assume we found a step. If we assume a particular form of
analyticity for the entries of $S_p$ (in Assumption 3), then we can
guarantee that there are only finitely many choices of the step length
$t_k$ for which $D\Lambda_{p^{(k+1)}}$ is not injective. In the
unlikely event one encounters one of such points, the step length $t_k$ can
be reduced by a small amount to make $D\Lambda_{p^{(k+1)}}$ injective.
This is a consequence of the following lemma, which is a
generalization of the result for the scalar discrete conductivity
inverse problem in \cite[Corollary 5.7]{Boyer:2015:SDC}.
\begin{lemma}
Consider a discrete inverse problem satisfying assumptions 1--3 and
further assume that all entries of $S_p$ are rational functions of $p$
(of the form $P(p)/Q(p)$, where $P$ and $Q$ are polynomials). Let $p \in
R$ and $\delta p \in \complex^m$ and assume the Jacobian of $\Lambda_p$
at $p$ is injective. Then there are at most finitely many $t \in \real$
for which $p + t \delta p \in R$ and either
\begin{enumerate}[(i)]
 \item The Jacobian of $\Lambda_p$ at $p+t\delta p$ is not injective.
 \item $\Lambda_{p+t\delta p} = \Lambda_p$.
\end{enumerate}
\end{lemma}
\begin{proof}
Since the Jacobian of $\Lambda_p$ is injective at $p \in R$, there is a
multi-index $\alpha \in \{1,\ldots,n^2\}^m$ such that the function $f$
defined in \eqref{eq:detfun} satisfies $f(p) \neq 0$. Since the
admissible set is open and convex, there is an interval $[a,b]$
containing $0$ such that $t \in [a,b] \implies p + t \delta p \in R$.
Since the entries of $S_p$ are rational functions of $p$ and $f$ is
defined through a determinant we can see that the function $g(t) = f(p+t
\delta p)$ is a rational function of $t$, i.e. it can be written in the
form $g(t) = F(t)/G(t)$ where $F(t)$ and $G(t)$ are polynomials. Since
$F(t)$ can only have finitely many zeroes, we conclude that there are
only finitely many $t$ for which $g(t) = 0$, or in other words, for
which the Jacobian at $p + t \delta p$ is not injective. This proves
(i). To prove (ii) we consider the function $h(t) = \det [W(p,p+t\delta
p)]_{:,\alpha}$, with $\alpha$ being the same multi-index as in (i).
The function $h$ is also a rational function in $t$ with finitely many
zeroes. Notice that $h(t) \neq 0$ implies the  matrix $W(p,p+t\delta p)$
has full row rank. Using the boundary/interior identity \eqref{eq:intid}
we see that 
\[
 \Lambda_{p+t\delta p} - \Lambda_p = t W(p,p+t\delta p)^T \delta p \neq 0
\]
when $\delta p \neq 0$. Thus there are at most finitely many $t$ for
which $ \Lambda_{p+t\delta p} = \Lambda_p$, $p,p + \delta p \in R^2$ and
$\delta p \neq 0$.
\end{proof}

\begin{remark}
The assumption on the entries of $S_p$ being rational functions of $p$
is satisfied by all the examples of discrete inverse problems on graphs
that we consider in \cref{sec:matrixval:ex,sec:elnets}. This is a simple
consequence of the cofactor formula for the inverse of a matrix.
\end{remark}

\subsubsection{Probabilistic interpretation of uniqueness almost
everywhere}
\label{sec:proba}
The discussion in \cref{sec:uniqueness:ae} has a probabilistic flavor as was remarked for the scalar
conductivity problem in \cite{Boyer:2015:SDC}. To see this, consider a
probability space $(\Omega,\cF,\proba)$ (i.e. a sample space $\Omega$, a
set of events $\cF$ and a probability measure $\proba$) and consider a random variable
$P : \Omega \to R\times R$ with distribution $\mu_P$ that we assume is
absolutely continuous with respect to the Lebesgue measure on
$\complex^m \times \complex^m$. Note that this assumption precludes the
distribution $\mu_P$ from being supported on a set of Lebesgue measure
zero in $R \times R$. We write $P \equiv (P_1,P_2)$ when we
want to distinguish the components of $P$. If uniqueness a.e. holds for the
discrete inverse problem at hand and $M \subset R \times R$ is a
measurable set for which $\proba \{ P \in M \} > 0$, then we must have
\[
 \proba \{ W(P_1,P_2)~\text{is injective} ~|~ P\in M \} = 1.
\]
To see this, remark that uniqueness a.e. guarantees that the set
\[
 Z = \{ (p_1,p_2) \in M ~|~ W(p_1,p_2)~\text{is not injective} \}
\]
is of measure zero. Since the distribution is absolutely continuous with
respect to the Lebesgue measure, this also means $\mu_P(Z) = 0$. Roughly
speaking, if we choose two admissible parameters $p_1,p_2$ at random, we
have $W(p_1,p_2)$ injective almost surely. Thus we can tell
$p_1$ and $p_2$ apart from the data $\Lambda_{p_1}$, $\Lambda_{p_2}$
almost surely.

A similar observation can be made regarding the injectivity of the
Jacobian of the problem. Let $Q : \Omega \to R$ be a random variable
with distribution $\mu_Q$ that is assumed to be absolutely continuous
with respect to the Lebesgue measure. If uniqueness a.e. holds and $N
\subset R$ is some measurable set with $\proba \{ Q \in N \} > 0$, then
we must have 
\[
 \proba \{ \text{Jacobian at $Q$ is injective} ~|~ Q\in N \} = 1.
\]

\section{The matrix valued conductivity and Schr\"odinger problems}
\label{sec:matrix:problems}

\subsection{Notation} 
\label{sec:notation}
We use the set theory notation $Y^X$ for the set of functions from $X$
to $Y$. For example $u \in (\complex^d)^X$ is a function $u: X \to
\complex^d$ that to some $x \in X$ associates $u(x) \in \complex^d$. For
some matrix $a \in \complex^{d \times d}$ we write $a \succ 0$ (resp. $a
\succeq 0$) to say that $a$ is positive definite (resp. positive
semidefinite). When the same notation is used for $a \in
(\complex^{d\times d})^X$, the generalized inequality is understood
componentwise, e.g. for $a \in (\complex^{d\times d})^X$, $a \succ 0$
means $a(x) \succ 0$ for all $x \in X$. When we write $a\succ b$ (or $a
\succeq b$) we mean $a - b\succ 0$ (or $a-b \succeq 0$). We use the
notation $a = a' + \jmath a''$, for the real $a' = \Re a$ and imaginary
$a'' = \Im a$ parts of $a$.

By ordering a finite set $X$, it can be identified with
$\{1,\ldots,|X|\}$, where $|X|$ is the cardinality of $X$. Thus
$(\complex^d)^X$ can be identified with vectors in $\complex^{d|X|}$.
Similarly, upon fixing an ordering for another finite set $Y$, we can
identify linear operators $(\complex^d)^X \to (\complex^d)^Y$ with
matrices in $\complex^{d|X| \times d|Y|}$.

For $A \in \complex^{m \times n}$, we denote by $\vect(A) \in
\complex^{mn}$ the {\em vector representation} of the matrix $A$, i.e.
the vector obtained by stacking the columns of $A$ in their natural
ordering.  Similarly for $a \in (\complex^{d \times d})^X$, we denote by
$\vect(a) \in \complex^{d^2|X|}$, the {\em vector representation} of
$a$, is the vector obtained by stacking the vector representations
$\vect(a(x))$ of the matrices $a(x)$, for $x \in X$ in the predetermined
ordering of $X$. 

In addition to the usual matrix vector product, we also use a block-wise outer product
($\outter$), the Hadamard product ($\hadamard$) and the Kronecker ($\kron$) product.
For $u,v \in (\complex^d)^X$, the (block-wise) {\em outer product} $u \outter v \in
(\complex^{d\times d})^X$  is 
\begin{equation}
 \label{eq:outer}
 (u \outter v)(x) = u(x) v(x)^T, ~ x \in X.
\end{equation}
The Hadamard or
componentwise product of two vectors $a, b \in \complex^n$ is denoted by
$a \hadamard b$ and it is given by $(a\hadamard b)(i) = a(i)b(i)$,
$i=1,\ldots n$. Finally the Kronecker product of two matrices $A \in
\complex^{n\times m}$ and $B \in \complex^{p \times q}$ is
the $np \times mq$ complex matrix $A \kron B$ given by
(see e.g.  \cite{Horn:2013:MA})
\begin{equation}
\label{eq:kron}
 A \kron B = 
 \begin{bmatrix}
  A_{11} B & \ldots & A_{1 m} B\\
  \vdots & & \vdots\\
  A_{n 1} B & \ldots & A_{n m} B
 \end{bmatrix}.
\end{equation}

\subsection{Discrete gradient, Laplacian and Schr\"odinger operators}
We work with graphs $G = (V,E)$, where $V$ is the set of vertices or nodes 
(assumed finite) and $E$ is the set of edges $E \subset \{ \{i,j\} |
i,j \in V, i \neq j\}$. All graphs we consider are 
undirected and with no self-edges. We partition the nodes $V = B
\cup I$ into a (nonempty) set $B$ of ``boundary'' nodes and a set $I$ of
``interior'' nodes.

By (discrete) {\em conductivity} $\sigma$ we mean a symmetric matrix
valued function defined on the edges, i.e. $\sigma \in (\complex^{d
\times d})^E$.  Here symmetric means $[\sigma(e)]^T = \sigma(e)$, for all
$e \in E$.  By (discrete) {\em Schr\"odinger potential} $q$ we mean a
symmetric matrix valued nodal function i.e. $q \in (\complex^{d\times
d})^V$.

The $d-$dimensional {\em discrete gradient} is the linear
operator $\nabla :(\complex^d)^V \to (\complex^d)^E$ defined for some $u
\in (\complex^d)^V$ by 
\[ 
 (\nabla u)(\{i,j\}) = u(i)-u(j), ~\{i,j\} \in E. 
\]
The discrete gradient assumes an edge orientation that
is fixed a priori and that is irrelevant in the remainder of this paper.

The {\em weighted graph Laplacian} is the linear map
$L_\sigma : (\complex^d)^V \to (\complex^d)^V$ defined by
\begin{equation}\label{eq:wgl}
L_\sigma = \nabla^T \diag(\sigma) \nabla,
\end{equation}
where we used the linear operator $\diag(\sigma) : (\complex^d)^E \to
(\complex^d)^E$, which is defined for some $v \in (\complex^d)^E$ by
$(\diag(\sigma) v)(e) = \sigma(e) v(e)$, $e \in E$. Its matrix
representation is a block diagonal matrix with the $\sigma(e)$ on its
diagonal. The operator $\nabla^T : (\complex^d)^E \to (\complex^d)^V$ is
the adjoint of the $d$-dimensional discrete gradient $\nabla$.

The discrete {\em Schr\"odinger
operator} associated with a conductivity $\sigma$ and a Schr\"odinger
potential $q$ is a block diagonal perturbation (with
blocks of size $d \times d$) of the weighted graph Laplacian, i.e
\[
 L_\sigma + \diag(q).
\]

\begin{example}
\label{ex:cyl}
We show how to use the matrix valued conductivities and Schr\"odinger
potentials to view the Laplacian of a cylindrical graph $C = P_k \times
G$ with scalar weights as a matrix valued Schr\"odinger operator on
the graph $P_k$, a path with $k$ nodes with vertices $V(P_k) =
\{1,\ldots,k\}$ and edges $E(P_k) = \{\{1,2\},\ldots,\{k-1,k\}\}$. Here
$\times$ denotes the Cartesian product between graphs. Such cylindrical
graphs arise e.g. in a finite difference discretization of the
conductivity equation on a rectangle with a Cartesian grid, as illustrated
in \cref{fig:cyl}. 

Let $s \in (0,\infty)^{E(C)}$ be a scalar conductivity on the
cylindrical graph $C$. We view $s$ as a vector and split it into the
sub-vectors $s^{j} \in
(0,\infty)^{E(G)}$, $j=1,\ldots,k$ and $s^{j,j+1} \in
(0,\infty)^{V(G)}$, $j=1,\ldots,k-1$.  The sub-vector $s^{j}$ represents the
scalar conductivity of the $j$-th copy of the graph $G$. The sub-vector
$s^{j,j+1}$ corresponds to the conductivity linking layer $j$ to layer
$j+1$. Define the matrix valued conductivity $\sigma \in (\real^{|V(G)|
\times |V(G)|})^{E(P_k)}$ by $\sigma(\{j,j+1\}) = \diag s^{j,j+1}$,
$j=1,\ldots,k-1$ and matrix valued Schr\"odinger potential $q \in (\real^{|V(G)|
\times |V(G)|})^{V(P_k)}$ by $q(j) = L_{s^j}$, i.e. the Laplacian of the
graph induced by the vertices in the $j-$th copy of $G$,
$j=1,\ldots,k$. Then with an appropriate ordering of the vertices we have
\[
 L_\sigma + \diag(q) = L_s.
\]

\end{example}

\begin{figure}
 \begin{center}
 \includegraphics[width=0.7\textwidth]{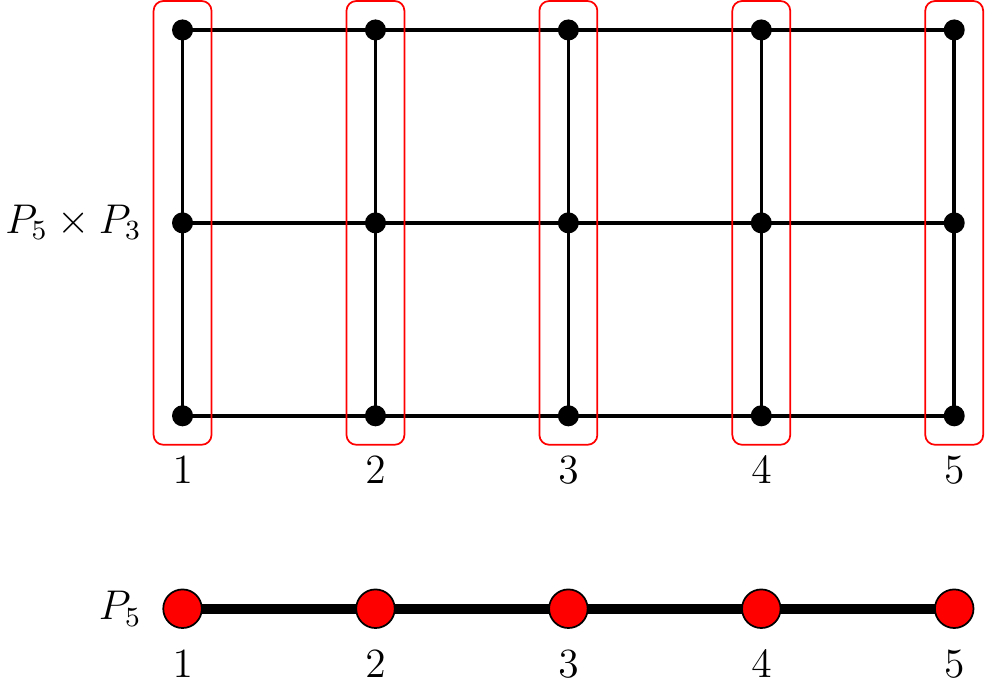}
 \end{center}
 \caption{The Laplacian for a scalar conductivity on the cylindrical
 graph $C \equiv P_5 \times P_3$ can be seen as a matrix valued
 Schr\"odinger operator on the graph $P_5$, as explained in
 \cref{ex:cyl}. To fix ideas, $s^4 \in \real^2$ represents the
 conductivities of $C$ within the $4-$th group in red and defines the
 matrix valued Schr\"odinger potential $q(4)$. The conductivity $s^{2,3}
 \in \real^3$ represents the conductivities of the 3 edges between the 2nd
 and 3rd group and is used to define the matrix valued conductivity
 $\sigma(\{2,3\})$.}
 \label{fig:cyl}
\end{figure}

\subsection{The Dirichlet problem}

For a conductivity $\sigma \in (\complex^{d \times d})^E$ and a
Schr\"odinger potential $q \in (\complex^{d \times d})^V$, the
$\sigma,q$ Dirichlet problem consists in finding $u \in (\complex^d)^V$
satisfying
\begin{equation}\label{eq:dir}
 \left\{
 \begin{aligned}
  ((L_\sigma + \diag(q))u)_I &=0,~\text{and}\\
  u_B &=g,
 \end{aligned}
 \right.
\end{equation}
where $g \in (\complex^d)^B$ is the Dirichlet boundary condition. The
Dirichlet to Neumann map, when it exists, is the linear mapping
$\Lambda_{\sigma,q} : (\complex^d)^B \to (\complex^d)^B$ defined by
\begin{equation}\label{eq:dtn}
 \Lambda_{\sigma,q} g = ((L_\sigma + \diag(q)) u)_B,
\end{equation}
where $u$ solves the Dirichlet problem \cref{eq:dir} with boundary
condition $u_B = g \in (\complex^d)^B$. The Dirichlet to Neumann map is
well defined e.g. when the solution to the Dirichlet
problem is uniquely determined by the boundary condition.\footnote{In
\cref{sec:sigpsd} we consider Dirichlet problems that do not admit a
unique solution and yet the Dirichlet to Neumann map is well defined.} Conditions
guaranteeing Dirichlet problem uniqueness are given in the next
theorem.
\begin{theorem}\label{thm:dir:pd}
The $\sigma, q$ Dirichlet problem on a connected graph with connected
interior admits a unique solution when $\sigma \in (\complex^{d\times
d})^E$ and $q \in (\complex^{d \times d})^V$ are symmetric and one of
the two following conditions is satisfied.
\begin{enumerate}[(i)]
 \item $\sigma' \succ 0$ and $q_I' \succ -
 \lambda_{\min}\M{(L_{\sigma'})_{II}}$.
 \item $q_I' \succ 0$ and $(L_{\sigma'})_{II} \succ -
 \lambda_{\min}(\diag(q_I'))$.
\end{enumerate}
\end{theorem}
In the previous theorem, $\lambda_{\min}(A)$ denotes the smallest
eigenvalue of a real symmetric matrix $A$. When any of the
conditions from \cref{thm:dir:pd} hold, the
Dirichlet to Neumann map can be written as
\begin{equation}
\label{eq:dtnmap}
 \Lambda_{\sigma,q} = L_{BB} + \diag(q_B) - L_{BI}(L_{II} + \diag(q_I))^{-1} L_{IB},
\end{equation}
where we dropped the subscript $\sigma$ in the blocks $(L_\sigma)_{BB}$,
$\ldots$ for  clarity.

Unfortunately \cref{thm:dir:pd} and the expression
\cref{eq:dtnmap} of the Dirichlet to Neumann map do not apply to one
the main applications of our results: static
spring networks. As we see in more detail in \cref{sec:springnets}, the 
linearization of Hooke's law we use allows for non-physical {\em floppy
modes}, i.e.  non-zero displacements that can be made with zero forces.
A generalization of the static spring network problem is to consider
symmetric conductivities with $\sigma \succeq 0$. In this
situation, floppy modes may also arise if there are edges $e$ for which
$\sigma(e)$ has a non-trivial nullspace. They can be defined as follows.

\begin{definition}\label{def:floppy} 
A non-zero $z \in \complex^V$, is said to be a {\em floppy
mode} for a symmetric conductivity $\sigma \in (\complex^{d\times d})^E$
with $\sigma \succeq 0$ if $z$ solves the equation
\begin{equation}\label{eq:floppy:def}
\left\{
\begin{aligned}
 (L_\sigma z)_I &=0,\\
 z_B & =0.
\end{aligned}
\right.
\end{equation}
\end{definition}

If $z$ is a floppy mode, then the solution to the $\sigma,0$ Dirichlet
problem cannot be unique. Indeed if $u$ is a solution to the $\sigma,0$
Dirichlet problem, them so is $u + \alpha z$ for any scalar $\alpha$.
The following theorem shows that even in the degenerate case $\sigma
\succeq 0$, $q=0$, there are situations where the Dirichlet problem
admits a solution that is unique up to floppy modes.
\begin{theorem}\label{thm:dir:psd}
The $\sigma, 0$ Dirichlet problem on a connected graph with connected
interior and $\sigma(e) \neq 0$ for all $e \in E$, admits a unique
solution up to floppy modes when any of the two following conditions
hold.
\begin{enumerate}[(i)]
 \item $\sigma' \succeq 0$ and $\sigma''=0$.

 \item $\sigma' \succeq 0$, and for each $e \in E$, $\sigma''(e)$ commutes with $\sigma'(e)$
 with nullspaces satisfying the inclusion $\nullspace(\sigma'(e)) \subset \nullspace(\sigma''(e))$.

\end{enumerate}
\end{theorem}

The next lemma shows that the Dirichlet to Neumann map is well defined
in the degenerate cases considered in \cref{thm:dir:psd}.
\begin{lemma}
\label{lem:dtn:psd}
Under the hypothesis of \cref{thm:dir:psd}, the Dirichlet to
Neumann map is 
\begin{equation}
 \label{eq:dtn:psd}
 \Lambda_{\sigma,0} = L_{BB} - L_{BI} Q (Q^T L_{II} Q)^{-1} Q^T L_{IB},
\end{equation}
where for clarity we dropped the subscript $\sigma$ in the blocks
$(L_\sigma)_{BB}$ etc. The matrix $Q$ is real with orthonormal columns
($Q^T Q = \identity$) and satisfies $\range(Q) = \range(L_{II})$.
Moreover $Q$ depends only on the eigenvectors of $\sigma'(e)$ associated
with non-zero eigenvalues, for $e \in E$.
\end{lemma}

The proofs of \cref{thm:dir:pd,thm:dir:psd,lem:dtn:psd} are
deferred to \cref{sec:diruni}.

\begin{remark}[Discrete Dirichlet principle]
For real $\sigma \succeq 0$ and $q \succeq 0$, it is easy to show that
the Dirichlet problem \eqref{eq:dir} is equivalent to finding $u \in
(\real^d)^V$ minimizing the energy
\begin{equation}
 \begin{aligned}
 E(u) = u^T (L_\sigma + \diag(q)) u =& 
   \sum_{\{i,j\} \in E} (u(i) - u(j))^T \sigma(\{i,j\}) (u(i)-u(j))\\
 &+ \sum_{k \in V} u(k)^T q(k) u(k),
 \end{aligned}
\label{eq:energy}
\end{equation}
subject to $u_B = g$. The function $E(u)$ is the energy needed to
maintain a potential $u$ in the network and is the sum of energies
associated to each edge and node.  The edge terms are akin to the
current-voltage product to calculate the power dissipated by a two
terminal electrical component. The nodal terms represent the energy leaked by
an electrical component  linking the node to the ground (zero
potential).  The conditions $\sigma\succeq 0$, $q \succeq 0$ guarantee $E(u)$ is a
convex quadratic function in $u$. The first equality in the Dirichlet
problem \eqref{eq:dir} identical to $\nabla_{u_I} E(u) = 0$.
\end{remark}

\subsection{Relating boundary and interior quantities}
\label{sec:interior}
The following lemma is a straightforward generalization to complex
matrix valued conductivities and Schr\"odinger potentials of the
interior identities \cite[Lemmas 5.1 and 6.1]{Boyer:2015:SDC}, which are
in turn inspired by the continuum interior identities used by Sylvester
and Uhlmann \cite{Sylvester:1987:GUT} to prove uniqueness for the
continuum conductivity and Schr\"odinger problems.

\begin{lemma}[Boundary/Interior Identity]
\label{lem:interior}
Let $\sigma_1, \sigma_2 \in (\complex^{d \times d})^E$ be conductivities
and $q_1, q_2 \in (\complex^{d \times d})^V$ be Schr\"odinger
potentials.  Let $u_1, u_2 \in (\complex^d)^V$ be solutions to the
$\sigma_1,q_1$ and $\sigma_2,q_2$ Dirichlet problems:
\[
\left\{
\begin{aligned}
 ((L_{\sigma_1} + \diag(q_1)) u_1)_I &=0,\\ 
 (u_1)_B &= g_1,
\end{aligned}
\right.
\quad\text{and}\quad
\left\{
\begin{aligned}
 ((L_{\sigma_2} + \diag(q_2)) u_2)_I &= 0,\\
 (u_2)_B &= g_2,
\end{aligned}
\right.
\]
for some boundary conditions $g_1, g_2 \in (\complex^d)^B$. Then if the Dirichlet to Neumann maps
$\Lambda_{\sigma_i,q_i}$, $i=1,2$, are well defined we have the identities
\[
\begin{aligned}
 g_2^T(\Lambda_{\sigma_1,q_1} - \Lambda_{\sigma_2,q_2})g_1 
 &= u_2^T ( L_{\sigma_1} - L_{\sigma_2} ) u_1 
  + u_2^T\diag(q_1 -q_2) u_1\\
 &=\sum_{e\in E} [(\nabla u_2)(e)]^T [(\sigma_1 -
 \sigma_2)(e)] [(\nabla u_1)(e)]\\
 &+ \sum_{i \in V} [u_2(i)]^T [(q_1 - q_2)(i)] [u_1(i)]\\
 &= \vect( (Du_2) \outter (Du_1))^T \vect(\sigma_1 - \sigma_2)\\
 &+ \vect(u_2 \outter u_1) ^T \vect(q_1 - q_2), \\
 \end{aligned}
\]
where the outer product $\outter$ is as in \eqref{eq:outer}.
\end{lemma}
\begin{proof}
Since $u_1$ solves the $\sigma_1,q_1$ Dirichlet problem we have
\begin{equation}
 \begin{aligned}
 u_2^T L_{\sigma_1} u_1 +  u_2^T \diag(q_1) u_1 &=
 (u_2)_B^T ((L_{\sigma_1} + \diag(q_1))u_1)_B
 + (u_2)_I^T ((L_{\sigma_1} + \diag(q_1)) u_1)_I\\
 &= (u_2)_B^T ((L_{\sigma_1} + \diag(q_1))u_1)_B\\
 &= g_2^T
\Lambda_{\sigma_1,q_1} g_1.
 \end{aligned}
 \label{eq:green1}
\end{equation}
Similarly, we have that 
\begin{equation}
u_2^TL_{\sigma_2} u_1 + u_2^T \diag(q_2) u_1 = g_2^T\Lambda_{\sigma_2,q_2} g_1.
\label{eq:green2}
\end{equation}
Subtracting \cref{eq:green2} from \cref{eq:green1} gives the first
equality. To obtain the second equality, use the definition of the
weighted graph Laplacian to see that
\[
 u_2^T L_{\sigma_i} u_1 = \sum_{e\in E} [ (Du_2)(e)
 ]^T\sigma_i(e)(Du_1)(e), ~i=1,2.
\]
By applying for each $e\in E$ the identity $x^T A y = \vect(xy^T)^T
\vect(A)$, which holds for any $x,y \in \complex^d$ and $A \in
\complex^{d\times d}$, we get
\begin{equation}
 u_2^T (L_{\sigma_1} - L_{\sigma_2}) u_1 =
\sum_{e\in E} \vect( [(Du_2)(e)] [(Du_1)(e)]^T )^T \vect((\sigma_1 -
\sigma_2)(e)).
\label{eq:vect1}
\end{equation}
By applying the same identity for all nodes $i \in V$ we get
\begin{equation}
 u_2^T \diag(q_1 - q_2) u_1 = \sum_{i \in V} \vect([u_2(i)] [u_1(i)]^T)^T
 \vect((q_1-q_2)(i)).
 \label{eq:vect2}
\end{equation}
The third equality follows from identities \cref{eq:vect1} and \cref{eq:vect2}.
\end{proof}

\section{Dirichlet problem uniqueness proofs}
\label{sec:diruni}

We first focus on cases where the solution to the $\sigma,q$ Dirichlet
problem is unique, either because $\sigma' \succ 0$ (\cref{sec:sigpd})
or because $q_I' \succ 0$ (\cref{sec:qpd}). In both cases the objective
is to show that the conditions given in \cref{thm:dir:pd} are
sufficient to guarantee that the matrix $(L_\sigma)_{II} + \diag(q_I)$ is
invertible.  The case where $\sigma' \succeq 0$, $q=0$ is dealt with in
\cref{sec:sigpsd}, and is more delicate because the matrix
$(L_\sigma)_{II} + \diag(q_I)$ is no longer invertible. However it is
still possible to show that the $\sigma,0$ Dirichlet solution is unique up
to floppy modes (\cref{def:floppy}).

\subsection{Conductivities with
positive definite real part} \label{sec:sigpd}

The goal of this section is to show that $(L_\sigma)_{II} + \diag(q_I)$ is
invertible when $\sigma' \succ 0$ and $q_I' \succ -\zeta$, for some
$\zeta>0$ to be determined and depending on $\sigma'$. To achieve this
we need two intermediary results on the discrete graph Laplacian
$L_\sigma$ with real matrix valued symmetric conductivity $\sigma \succ
0$.  The first one is a discrete version of the first Korn inequality
(\cref{lem:korn}). The second is to show that a vector potential $u
\in \nullspace(L_\sigma)$ must be constant on all connected
components of the graph $G$ (\cref{lem:const}).  Using these
properties, we can show that when $\sigma$ is a real conductivity with
$\sigma \succ 0$, we have $(L_\sigma)_{II} \succ 0$. This establishes
uniqueness for the $\sigma,0$ Dirichlet problem for real $\sigma$ with
$\sigma \succ 0$.  The extension to complex conductivities and non-zero
Schr\"odinger potentials (\cref{lem:spd}) follows from studying the
field of values (see e.g. \cite{Horn:2013:MA}) of the sum of a symmetric
positive definite real matrix and a purely imaginary symmetric matrix
(\cref{lem:fov}).

The following is a discrete version of the first Korn inequality which 
bounds the elastic energy stored in a body from below by the gradient of
the strain, see e.g. \cite[\S 1.12]{Marsden:1994:MFE}.
\begin{lemma}[Discrete Korn inequality]
\label{lem:korn}
Let $\sigma \in (\real^{d\times d})^E$ be a conductivity with $\sigma
\succ 0$.  Then there is a constant $C>0$ such that for any $u
\in (\real^d)^V$,
\begin{equation}
\| \nabla u\|^2 \leq C u^T L_\sigma u.
\end{equation}
\end{lemma}
\begin{proof}
By using Rayleigh quotients,
\[
  v ^T \sigma(e) v \geq \lambda_{\min}(\sigma(e)) \|v\|^2 ~\text{for all}~ v\in
  \real^d ~\text{and}~e \in E.
\]
Define $\lambda_* = \min_{e\in E} \lambda_{\min}(\sigma(e)) =
\lambda_{\min}(\diag(\sigma))$. Clearly $\sigma
\succ 0$ implies $\lambda_* > 0$. The inequality we seek follows with $
C = \lambda_*^{-1}$ from
\[
 u ^T L_\sigma u =  \sum_{e \in E} [(\nabla u)(e)]^T [\sigma(e)]
 [(\nabla u)(e)] \geq \lambda_* \sum_{e \in E} \|(\nabla u)(e)\|^2 =
 \lambda_* \| \nabla u\|^2.
\]
\end{proof}

The next lemma extends to matrix valued conductivities a well known
characterization of the nullspace of (scalar) weighted graph
Laplacians (see e.g. \cite{Chung:1997:SGT}).

\begin{lemma}[Nullspace of graph Laplacian]
\label{lem:const}
For real $\sigma \succ 0$, $u \in \nullspace(L_\sigma)$ implies that $\nabla u = 0$.
In particular if the graph is connected then $u$ is constant, meaning
there is a constant $c \in \real^d$ such that $u(i) = c$ for all $i\in V$.
\end{lemma}
\begin{proof}
If $u \in \nullspace(L_\sigma)$ then $u ^T L_\sigma u = 0$. Using the discrete
Korn inequality (\cref{lem:korn}), we get $\nabla u = 0$. This
means that for any edge $\{i,j\} \in E$, we must have $u(i) = u(j)$.
Therefore $u$ must be constant on connected components of the graph.
\end{proof}

We can now prove the first uniqueness result for the Dirichlet problem.
\begin{lemma}[Uniqueness for real positive definite
conductivities]\label{lem:dir} Assume both the graph $G$ and its
subgraph induced by the interior nodes are connected. For real
conductivities $\sigma$ with $\sigma \succ 0$, the matrix
$(L_\sigma)_{II}$ is invertible and the $\sigma,0$ Dirichlet problem
admits a unique solution.
\end{lemma}
\begin{proof}
Our goal here is to show that $(L_\sigma)_{II} \succ 0$ which implies invertibility and
therefore uniqueness for the $\sigma,0$ Dirichlet problem. By definition
of the weighted graph Laplacian \cref{eq:wgl}, the matrix $L_\sigma$
must be real and symmetric. Moreover using the discrete Korn
inequality (\cref{lem:korn}), there is a constant
$C > 0$ such that for all $u \in (\real^d)^V$:
\[
u^T L_\sigma u \geq C \|\nabla u\|^2.
\]
This implies $L_\sigma \succeq 0$ and hence $(L_\sigma)_{II} \succeq 0$. Now we can write
\[
 (L_\sigma)_{II} = L_{\sigma_I} + \diag(f),
\]
where $L_{\sigma_I}$ is the weighted graph Laplacian on the subgraph of
$G$ induced by the interior nodes $I$ and $f \in (\real^{d\times
d})^I$ is given for $i \in I$ by
\begin{equation}
\label{eq:fi}
f(i) = \sum_{\{i, j\} \in E, i \in I, j \in B} 
\sigma(\{i, j\}).
\end{equation}
Since the sum of positive definite matrices is positive definite,
$\sigma \succ 0$ implies $f(i) \succ 0$ for all nodes $i\in I$ that are
connected via an edge to some boundary node and $f(i) = 0$ otherwise.
This guarantees that $\diag(f) \succeq 0$.

Now take $v \in (\real^d)^I$ with $v^T (L_{\sigma})_{II} v =0$.
Since both $L_{\sigma_I}$ and $\diag(f)$ are positive semidefinite, we
must have that 
\[
\begin{aligned}
 \text{(a)} & ~ v^T L_{\sigma_I} v = 0
 ~\text{and}~\\
 \text{(b)} &~ v^T \diag(f) v = 0.
\end{aligned}
\]
By using (a) and \cref{lem:const} on the subgraph
induced by the interior nodes (which is connected by assumption), we get
that $v$ is constant, i.e.  $v(i)=v(j)$ for any $i,j \in V$. By using
(b), we get that $v(i)^T \sigma(\{i, j\})  v(i) = 0$ for all $\{i, j\}
\in E$ where $i \in I$ and $j \in B$. Hence there must be at least one
$i \in I$ such that $v(i)= 0$ (since $G$ is connected). Since the
subgraph of $G$ induced by the interior nodes is connected, we conclude
that $v=0$. This gives the desired result $(L_\sigma)_{II}\succ 0$.
\end{proof}

The following lemma allows us to extend the uniqueness result from
\cref{lem:dir} to complex conductivities and Schr\"odinger
potentials.
\begin{lemma}
\label{lem:fov}
Let $A,B \in \real^{n \times n}$ be symmetric with $A \succ 0$. Then the
matrix $M = A + \jmath B$ is invertible.
\end{lemma}
\begin{proof}
The field of values (or numerical range, see e.g. \cite{Horn:2013:MA})
of $M \in \complex^{n \times n}$ is the complex plane region given by
\[
F(M) = \Mcb{ v^*M v ~|~ v \in \complex^n,~\|v\|=1}.
\]
Since $A \succ 0$ we have $v^* A v >0$ for $v \neq 0$. Since $B$ is real
symmetric, $v^*B v$ must be real. Therefore $\Re(v^* M v) = v^*Av > 0$
for $v\neq 0$ and the
field of values $F(M)$ lies on the right hand complex plane, excluding
the imaginary axis. Since the spectrum of $M$ is contained in $F(M)$
this means that $0$ is not an eigenvalue of $M$ and that $M$ is
invertible.
\end{proof}

We are now ready to show that the condition (i) in
\cref{thm:dir:pd} is sufficient for having a unique solution to
the $\sigma,q$ Dirichlet problem.
\begin{lemma}
\label{lem:spd}
Let $G$ be a connected graph with connected subgraph induced by the
interior nodes. Let $\sigma \in (\complex^{d\times d})^E$ be a
conductivity with $\sigma' \succ 0$ and $q \in (\complex^{d\times d})^V$
be a Schr\"odinger potential with $q_I' \succ
-\lambda_{\min}\M{(L_{\sigma'})_{II}}$. Then the matrix $(L_\sigma)_{II}
+ \diag(q_I)$ is invertible and the $\sigma,q$ Dirichlet problem admits a
unique solution.
\end{lemma}
\begin{proof}
 By \cref{lem:dir} and because $\sigma' \succ 0$, we have that
 $(L_{\sigma'})_{II} \succ 0$. Since we assume $q_I' \succ
-\lambda_{\min}\M{(L_{\sigma'})_{II}}$, we must have $(L_{\sigma'})_{II} +
\diag(q_I') \succ 0$. We can now use \cref{lem:fov} with $A \equiv
(L_{\sigma'})_{II} + \diag(q_I')$ and $B \equiv (L_{\sigma''})_{II} +
\diag(q_I'')$ to conclude that $(L_\sigma)_{II} + \diag(q_I)$ is invertible.
Uniqueness follows from the definition of the $\sigma, q$ Dirichlet problem.
\end{proof}

\subsection{Schr\"odinger potentials
with positive definite real part}\label{sec:qpd}
The main result of this section is the following lemma, which shows that
$q_I' \succ 0$  and a condition on the smallest eigenvalue of
$(L_{\sigma'})_{II}$ (i.e. condition (ii) in \cref{thm:dir:pd})
guarantees uniqueness for the $\sigma,q$-Dirichlet problem.
\begin{lemma}
\label{lem:qpd}
Let $G$ be a connected graph with connected subgraph induced by the
interior nodes. Let  $q \in (\complex^{d \times d})^V$ be a Schr\"odinger
potential with $q_I' \succ 0$ and $\sigma \in (\complex^{d\times d})^E$ be a
conductivity with $(L_{\sigma'})_{II} \succ -\lambda_{\min}(\diag(q_I'))$. Then
the matrix $(L_\sigma)_{II} + \diag(q_I)$ is invertible and the $\sigma,q$
Dirichlet problem admits a unique solution.
\end{lemma}
\begin{proof}
By the hypothesis, we have that $(L_{\sigma'})_{II} + \diag(q_I') \succ 0$.
Hence we can use \cref{lem:fov} with $A \equiv (L_{\sigma'})_{II} +
\diag(q_I')$ and $B \equiv (L_{\sigma''})_{II} + \diag(q_I'')$ to conclude
that $(L_\sigma)_{II} + \diag(q_I)$ is invertible and the desired result follows.
\end{proof}

\subsection{Conductivities with positive semidefinite real part and zero
Schr\"odinger potential}
\label{sec:sigpsd}
The purpose of this section is to prove \cref{thm:dir:psd}, which
deals with a situation that is not covered by 
\cref{lem:spd,lem:qpd} because the Schr\"odinger
potential $q=0$ and the conductivity $\sigma' \succeq 0$. The discrete
Korn inequality (\cref{lem:korn}) does not apply in this situation,
but can be easily modified to avoid floppy modes
(\cref{lem:mkorn}). We give a characterization of floppy modes
(\cref{lem:floppy:char}) that shows that floppy modes for real
$\sigma$ are entirely determined by the subspace $\nullspace(\diag(\sigma))$ and that they
do not affect the boundary data (\cref{lem:floppy:zerob}). This
allows us to prove uniqueness up to floppy modes
(\cref{thm:dir:psd}) and gives an expression for the Dirichlet to
Neumann map (\cref{lem:dtn:psd}). The generalization of these
results to complex conductivities satisfying condition (ii) in
\cref{thm:dir:psd} follows by noticing that this condition
ensures the fundamental subspaces of different blocks of $L_\sigma$ are
identical to those of $L_\sigma'$ (\cref{lem:psd:complex}).

The following results rely on the eigendecomposition of the matrices
$\sigma(e) \in \complex^{d \times d}$, $e \in E$. To simplify the
discussion, we assume that the conductivities of all edges have the same
rank $r \geq 1$, i.e. $\rank \sigma(e) = r$ for all $e \in E$. This
suffices for our application to elastodynamic networks where $r = 1$
(\cref{sec:springnets}).  The results of the present section can be
adapted to the case where the conductivities have a rank that may vary
with edge, as long as $\rank(\sigma(e)) \geq 1$ for all $e \in E$. Since
the conductivities we consider here satisfy condition (ii) in
\cref{thm:dir:psd}, the eigenvectors of $\sigma(e)$ are real and
we may define $x \in (\real^{d\times r})^E$ and $\lambda \in
(\complex^r)^E$ to write the eigendecomposition of each of the
conductivities i.e.
\begin{equation}
 \sigma(e) = [x(e)] [\diag(\lambda(e))] [x(e)]^T,~ e \in E,
 \label{eq:seigen}
\end{equation}
with $x(e)^Tx(e) = \identity$ being the $r \times r$ identity. By
hypothesis we must have $\lambda' > 0$. Condition (i) in
\cref{thm:dir:psd} means $\lambda''=0$, whereas condition (ii) imposes no
restriction on $\lambda''$.

\subsubsection{Real case}
The following is a slight generalization of the discrete Korn
inequality \cref{lem:korn}.
\begin{lemma}[Modified Discrete Korn Inequality]
\label{lem:mkorn}
Let $\sigma \in (\real^{d \times d})^E$ be a real conductivity with
$\sigma \succeq 0$. Then there exists constants $C_1, C_2 > 0$ such that
\begin{equation}
C_1 \|\diag(x)^T  \nabla u\|^2 \leq u^T L_{\sigma}u \leq C_2
\|\diag(x)^T \nabla u\|^2.
\end{equation}
for any $u \in (\real^d)^V$. Here $x \in (\real^{d \times r})^E$ is such
that $x(e)$ is an eigenvector matrix for the positive eigenvalues of
$\sigma(e)$ as in \cref{eq:seigen}.
\end{lemma}
\begin{proof}
For the lower bound notice that 
\[
 \sigma(e) \succeq \lambda_{\min'}(\sigma(e)) [x(e)][x(e)]^T,~e\in E, 
\]
where for a matrix $A \succeq 0$, $\lambda_{\min'}(A)$ denotes the
{\em smallest positive} eigenvalue of $A$. Hence we get the lower bound:
\[
\begin{aligned}
u ^T L_{\sigma}u 
&=\sum_{e \in E} [(\nabla u)(e)]^T [\sigma(e)] [(\nabla u)(e)]\\
&\geq \sum_{e \in E} \lambda_{\min'}(\sigma(e)) [(\nabla u)(e)]^T
[x(e)][x(e)]^T
[(\nabla u)(e)]\\
&\geq \lambda_{\min'}(\diag(\sigma)) \sum_{e \in E} \| [x(e)]^T [(\nabla u)(e)] \|^2
 = \lambda_{\min'}(\diag(\sigma)) \|\diag(x)^T  \nabla u\|^2.
\end{aligned}
\]
The upper bound follows similarly from the bound
\[
\sigma(e) \preceq \lambda_{\max}(\sigma(e)) [x(e)] [x(e)]^T, ~e \in E.
\]
\end{proof}

We now give a characterization of the floppy modes, that shows these
modes depend only on the subspaces $\range(\sigma(e))$ (or equivalently
$\nullspace(\sigma(e))$), for $e \in E$.

\begin{lemma} 
\label{lem:floppy:char}
For real conductivities $\sigma\succeq 0$, the following are equivalent.
\begin{enumerate}[(i)] 
\item $z$ is a floppy mode (i.e. it satisfies \cref{eq:floppy:def})
\item $z$ is a non-zero solution to
\begin{equation}\label{eq:floppy2}
\left\{
\begin{aligned}
 \diag(x)^T \nabla z &= 0,\\
 z_B&= 0,
\end{aligned}
\right.
\end{equation}
where $x \in (\real^{d \times r})^E$ are the eigenvector matrices for
$\sigma$, as in \cref{eq:seigen}.
\item $z$ is such that $z_B = 0$ and $z_I \in \nullspace( (L_\sigma)_{II} )$.
\end{enumerate}
\end{lemma}

\begin{proof}
\underline{(ii) $\implies$ (i)}. Assume that $z \neq 0$ satisfies
\cref{eq:floppy2}. Then clearly $\|\diag(x)^T \nabla z\|^2=0$ and by the
second inequality in \cref{lem:mkorn} we must have $z^T
L_\sigma z=0$. Since we have that
\[
 0 = z^T L_\sigma z = z_B ^T (L_\sigma z)_B+z_I ^T (L_\sigma
 z)_I,
\]
and $z_B = 0$, we also have that $z_I^T(L_\sigma z)_I=0$ and $z_I^T
(L_\sigma)_{II} z_I=0$.  Since $L_\sigma$ is symmetric we conclude that
$(L_\sigma)_{II}z_I=0$, i.e. $z$ is a floppy mode (by
\cref{def:floppy}).  

\underline{(i) $\implies$ (ii).} Now we assume that $z$ is a floppy mode,
i.e. it satisfies \cref{eq:floppy:def}. Clearly this leads to $z^T
 L_\sigma z  = 0$. By using the first inequality in \cref{lem:mkorn},
we conclude that $\diag(x)^T \nabla z =0$ and that $z$ satisfies
\cref{eq:floppy2}.

\underline{(i) $\iff$ (iii).} Since $z_B = 0$, we have $z^T L_\sigma z =
z_I^T(L_\sigma)_{II}z_I$. Thus $z$ satisfying \cref{eq:floppy:def} implies
$z_I^T(L_\sigma)_{II}z_I =0$ and thus $z_I \in \nullspace((L_\sigma)_{II})$.
Similarly if $z_I \in \nullspace((L_\sigma)_{II})$ and $z_B = 0$, then
$z^T L_\sigma z = 0$ and $L_\sigma z = 0$. It follows that
$(L_\sigma z)_I =0$.
\end{proof}

Next we continue with a technical result, which is a slight
generalization of the elastic network result \cite[Lemma
1]{Guevara:2011:CCS}.

\begin{lemma}\label{lem:inclusion}
Let $\sigma$ be a real conductivity with $\sigma \succeq 0$. Then we
have the inclusion 
\[
 \range((L_\sigma)_{IB}) \subset \range((L_\sigma)_{II}).
\]
\end{lemma}
\begin{proof}
Since $L_\sigma$ is symmetric, it is equivalent to prove 
$\nullspace((L_\sigma)_{IB}) \supset \nullspace((L_\sigma)_{II})$.  Let $z \in
\nullspace((L_\sigma)_{II})$. The extension by zeros of $z$ to the boundary is a
floppy mode. Proceeding as in the proof of \cref{lem:dir}, we write
\[
 (L_\sigma)_{II} = L_{\sigma_I} + \diag(f),
\]
where $L_{\sigma_I}$ is the weighted graph Laplacian on the subgraph
induced by the boundary nodes and $f \in (\real^{d\times d})^I$ is given
as in \cref{eq:fi}. Since $L_{\sigma_I} \succeq 0$ and $\diag(f)\succeq
0$ we get that $z^T  (L_\sigma)_{II} z = 0$ implies $z^T
\diag(f) z  = 0$. Since $f(i)$ is a sum of positive semidefinite
matrices, we must have that 
\[
 z(i) ^T \sigma(\{i,j\}) z(i)= 0 ~\text{for
all}~ \{i,j\} \in E,~\text{with}~ i \in I, j \in B.
\]
Since conductivities
are symmetric, this means that $\sigma(\{i,j\}) z(i) = 0$ for all $\{i,j\}
\in E$, with $i \in I$ and $j \in B$. Now from the definition of the
Laplacian we have
\[
 ((L_\sigma)_{BI} z)(j) = \sum_{\{i,j\} \in E, i\in I} \sigma(\{i,j\})
 z(i) = 0,~\text{for $j \in B$}.
\]
This shows the desired result.
\end{proof}

We are now ready to prove the result of \cref{thm:dir:psd} in the
real case.
\begin{proof}[Proof of \cref{thm:dir:psd}, condition (i).]
Let us first assume condition (i) of \cref{thm:dir:psd} holds, i.e.
that $\sigma$ is real and $\sigma \succeq 0$. If $u$ is a solution to
the $\sigma,0$ Dirichlet problem with boundary condition $g \in
(\real^d)^B$, then $u_B = g$ and
\begin{equation}\label{eq:int}
 (L_\sigma)_{IB} g + (L_\sigma)_{II} u_I = 0.
\end{equation}
The inclusion of \cref{lem:inclusion} guarantees that equation
\cref{eq:int} admits a solution for all $g \in (\real^d)^B$. The
general solution to \cref{eq:int} may be written as
\begin{equation}\label{eq:pinv}
 u_I = - ((L_\sigma)_{II})^\dagger (L_\sigma)_{IB} g  + z,
\end{equation}
where $z \in \nullspace((L_\sigma)_{II})$ and the symbol $\dagger$ is the
Moore-Penrose pseudoinverse. 
\end{proof}

\subsubsection{Complex case}
Our objective is to prove \cref{thm:dir:psd} condition (ii) holds.
This can be done by proving the following lemma.

\begin{lemma}
\label{lem:psd:complex}
Let $\sigma \in (\complex^{d\times d})^E$ be a conductivity as in
condition (ii) in \cref{thm:dir:psd}. Then we have
\begin{enumerate}[(i)]
\item $\nullspace((L_\sigma)_{II})=\nullspace((L_{\sigma'})_{II})$,
\item $\nullspace((L_\sigma)_{II})\subset \nullspace((L_{\sigma})_{BI})$,
\item $\range((L_\sigma)_{II})=\range((L_{\sigma'})_{II})$,
\item $\range((L_\sigma)_{II})\supset \range((L_{\sigma})_{IB})$.
\end{enumerate}
\end{lemma}

\begin{proof}

\underline{Proof of statement (i).} Let $z_I \in \nullspace((L_\sigma)_{II})$. Then
we have
\[
0 = z_I^* (L_\sigma)_{II} z_I = z_I^* (L_{\sigma'})_{II} z_I  +\jmath
z_I^*
(L_{\sigma''})_{II} z_I.
\]
Therefore $z_I \in \nullspace((L_{\sigma'})_{II})$. Now assume $z_I \in
\nullspace((L_{\sigma'})_{II})$. The extension $z$ of $z_I$ by zeroes on
$B$ must be a floppy mode and satisfies \cref{eq:floppy2}. Since
we can rewrite
\begin{equation}
 L_\sigma = \nabla^T \diag(x) \diag(\lambda' + \jmath \lambda'')
 \diag(x)^T \nabla,
 \label{eq:decomposition}
\end{equation}
it follows that $L_\sigma z = 0$ and that $z_I \in
\nullspace((L_\sigma)_{II})$.

\underline{Proof of statement (ii).} Let $z_I \in \nullspace((L_\sigma)_{II})$.
Similarly to the proof of (i), we have that the extension $z$ of $z_I$ by
zeroes on $B$ must satisfy $L_\sigma z =0$ and in particular
$(L_\sigma)_{BI} z_I = 0$.

\underline{Proof of statement (iii).} Apply statement (i) to the
conductivity $\overline{\sigma} = \sigma' -\jmath \sigma''$ and the
orthogonality of the fundamental subspaces of a complex matrix to get
$\range((L_{\overline{\sigma}})_{II}^*) =
\range((L_{\sigma'})_{II}^*)$. Using that $\sigma'$ and
$\sigma''$ are real symmetric this gives the desired result
$\range((L_\sigma)_{II}) = \range((L_{\sigma'})_{II})$.%

\underline{Proof of statement (iv).} Apply statement (ii) to the
conductivity $\overline{\sigma} = \sigma' -\jmath \sigma''$ and the
orthogonality of the fundamental subspaces of a complex matrix to get
$\range((L_{\overline{\sigma}})_{II}^*) \supset
\range((L_{\overline{\sigma}})_{BI}^*)$. Using that $\sigma'$ and
$\sigma''$ are symmetric this gives the desired result
$\range((L_\sigma)_{II})\supset \range((L_{\sigma})_{IB})$.%
\end{proof}

We can now complete the proof of \cref{thm:dir:psd}.

\begin{proof}[Proof of \cref{thm:dir:psd}, condition (ii).]
The proof follows as in the real case. The inclusion (iv) in
\cref{lem:psd:complex} implies that  the linear equation
\cref{eq:int} always has a solution regardless of the boundary
condition $g$. Solutions to the Dirichlet problem can be written with
the pseudoinverse as in \cref{eq:pinv}.
\end{proof}
\subsubsection{Dirichlet to Neumann map for rank deficient
conductivities}
\label{sec:dtn:pinv}
The following lemma shows that even if there are floppy modes, these do
not influence Neumann (or net current) measurements at the boundary. In
other words, floppy modes cannot be observed from boundary measurements.

\begin{lemma}
\label{lem:floppy:zerob}
Floppy modes correspond to zero boundary measurements.
\end{lemma} 
\begin{proof}
 We need to show that if $z \in (\complex^d)^V$ is a floppy mode, then
 we have zero fluxes at the boundary, i.e. $(L_\sigma z)_B = 0$. If $z$
 is a floppy mode then it satisfies \cref{eq:floppy:def}. In particular
 we have
 \[
  z^* L_\sigma z = z_B^* (L_\sigma z)_B + z_I^* (L_\sigma
  z)_I = 0,
 \]
 because $z_B=0$ and $(L_\sigma z)_I = 0$. Since $L_\sigma$ is a 
 symmetric matrix, we also have $L_\sigma z = 0$. This gives the desired
 result $(L_\sigma z)_B = 0$.
\end{proof}

We are now ready to show that the Dirichlet to Neumann map is well
defined for real conductivities $\sigma \succeq 0$ and $q = 0$, and by
extension to certain complex conductivities (\cref{lem:dtn:psd}).

\begin{proof}[Proof of \cref{lem:dtn:psd}]
Let $g \in (\complex^d)^B$ be a Dirichlet boundary condition for the
$\sigma,0$ Dirichlet problem. From the proof of
\cref{thm:dir:psd}, a solution $u \in (\complex^d)^V$ satisfies
$u_B = g$ and $u_I = -((L_\sigma)_{II})^\dagger (L_\sigma)_{IB} g + z$,
for some $z \in \nullspace((L_\sigma)_{II})$.  The fluxes at  the boundary
corresponding to such solution are:
\[
 (L_\sigma u)_B = (L_\sigma)_{BB} g - (L_\sigma)_{BI} ((L_\sigma)_{II})^\dagger
 (L_\sigma)_{IB} g + L_{BI} z.
\]
However the inclusion (ii) in \cref{lem:psd:complex} (or
\cref{lem:floppy:zerob}) guarantees that
$(L_\sigma)_{BI} z = 0$. Hence the
Dirichlet to Neumann map is uniquely defined and can be written as
\begin{equation}
\label{eq:dtnpinv}
\Lambda_{\sigma,0} = (L_\sigma)_{BB} - (L_\sigma)_{BI} ((L_\sigma)_{II})^\dagger
 (L_\sigma)_{IB}.
\end{equation}
Now let $Q$ be such that $Q^TQ = I$ and $\range(Q) = \range((L_\sigma)_{II})$.
We can always find a real $Q$ because of (iii) in
\cref{lem:psd:complex}, and it can be found e.g. with the QR
factorization or by computing the eigendecomposition of
$(L_{\sigma})_{II}$.  The space $\range(Q)$ is the orthogonal to the
interior components of floppy modes, and thus depends only on the
eigenvectors $x$ (as in \cref{eq:seigen}) of the $\sigma'(e)$, $e \in
E$, associated with non-zero eigenvalues (\cref{lem:floppy:char}, (ii)).
We can use $Q$ to write the pseudoinverse of $(L_\sigma)_{II}$ as
follows
\[
 (L_\sigma)_{II}^\dagger = Q(Q^T (L_\sigma)_{II}Q)^{-1}Q^T,
\]
and we get the alternate expression \cref{eq:dtn:psd} for the Dirichlet to
Neumann map.
\end{proof}

\section{Examples of matrix valued inverse problems on graphs}
\label{sec:matrixval:ex}

Here we use the graph theoretical results from \cref{sec:matrix:problems} to
give examples of matrix inverse problems on graphs that fit the mold of
\cref{sec:structure}.

\subsection{Matrix valued conductivity inverse problem}
Given a graph $G = (V,E)$ with boundary, the problem here is to find the
matrix valued conductivity $\sigma \in (\complex^{d\times d})^E$ from the
Dirichlet to Neumann map $\Lambda_{\sigma,0}$. We explain below why
this problem satisfies the assumptions of \cref{sec:structure}.
\begin{itemize}
 \item Here we take as {\bf admissible set}:
  \[
   R = \{ \sigma \in (\complex^{d \times
 d} )^E ~|~ \sigma = \sigma^T ~\text{and}~ \sigma' \succ 0 \}.
 \]
 This is an open convex set in  $(\complex^{d \times d} )^E$ which can
 be identified to an open convex subset of $\complex^{d^2|E|}$.  The
 {\bf forward map} is the map that to $\sigma \in R$ associates the
 Dirichlet to Neumann map $\Lambda_{\sigma,0} \in \complex^{d|B| \times
 d|B|}$.  This map is well defined for $\sigma \in R$ because of
 \cref{thm:dir:pd}. 

 \item By
  \cref{lem:interior} with $q_i = 0$, $i=1,2$, we have the
  {\bf boundary/interior identity}
 \[
  g_2^T(\Lambda_{\sigma_1,0} - \Lambda_{\sigma_2,0}) g_1 =
  b(S_{\sigma_2} g_2 , S_{\sigma_1} g_1  )^T\vect(\sigma_1 -
  \sigma_2). 
 \]
 Here we define $S_\sigma \in \complex^{d|E|\times d|B|}$ by its
 action on some $f \in \complex^{d|B|}$
 \[
  S_\sigma f = Du,
 \]
 where $u$ solves the Dirichlet problem \cref{eq:dir}, with $q=0$ and
 $u_B = f$. The bilinear map $b : \complex^{d|E|} \times
 \complex^{d|E|} \to \complex^{|E|d^2}$ is defined by $b(u,v) = u
 \outter v$, where the outer product $\outter$ is defined in
 \cref{sec:notation} and we are implicitly identifying
 $(\complex^d)^E$ with $\complex^{d|E|}$ and similarly for
 $(\complex^{d\times d})^E$ and $\complex^{|E|d^2}$.
 
 \item {\bf Analyticity assumption.} From \cref{lem:dir,lem:fov,lem:spd}
  the solution $u$ to the Dirichlet problem \cref{eq:dir} with $u_B =
  f$, $\sigma \in R$ 
  and $q=0$ is determined by $u_I = -(L_{II})^{-1} L_{IB} f$, where for clarity we
  omitted the subscript $\sigma$ from the graph Laplacian $L_\sigma$.
  Therefore the entries of $S_\sigma$ depend analytically on $\sigma$,
  for $\sigma \in R$.
\end{itemize}

\subsubsection{Relation between scalar and matrix valued conductivity
problems}
Here we show that if the Jacobian for the scalar conductivity problem on
a graph is injective at a conductivity $\sigma \in (0,\infty)^E$ then the
Jacobian for the matrix conductivity 
problem on the same graph but with conductivity of edge $e \in E$ given
by $\sigma(e) \identity
\in \complex^{d\times d}$ must also also be injective. Because of the
discussion in \cref{sec:uniqueness:ae}, this result shows that if
uniqueness a.e. holds for a scalar conductivity problem, then it must
also hold for the matrix valued problem. In particular 
uniqueness a.e. holds on the critical circular planar graphs that are
defined in \cite{Curtis:1994:FCC,Curtis:1998:CPG}. 
\begin{lemma}
 Let $G = (V,E)$ be a graph with boundary and let $s \in
 (0,\infty)^E$ be a scalar conductivity. Define the conductivity $\sigma
 \in (\real^{d \times d})^E$ by $\sigma(e) = s(e) \identity$ with
 $\identity$ being the $d \times d$ identity. Then if the Jacobian of the
 forward problem is injective for the scalar conductivity $s$, it must
 also be injective for the matrix valued conductivity $\sigma$. 
\end{lemma}
\begin{proof}
We need to show that
 \[
  \range(W(s,s)) = \real^{|E|} \implies
  \range(W(\sigma,\sigma)) = \real^{d^2|E|}.
\]
To this end, we first link the Laplacian $L_\sigma$ for the matrix valued
conductivity $\sigma$ is a $d|V| \times d|V|$ matrix to the Laplacian
for a graph $G_d = (V_d,E_d)$ that corresponds to having $d$ copies of
the graph $G$ without any edges between the copies, and each copy of $G$
having the same scalar conductivity $s$. To be more precise the vertex
set of $G_d$ is $V_d = V \times \{1,\ldots,d\}$, the edge set is 
\[
E_d = \{ \{ (v_1,\ell_1) , (v_2,\ell_2) \} \in V_d
\times V_d ~|~ \ell_1 = \ell_2 ~\text{and}~\{v_1,v_2\} \in E \}.
\]
Then with an appropriate ordering of $V_d$, we have $L_\sigma = L_{s_d}$,
where the conductivity $s_d \in (0,\infty)^{E_d}$ is defined by 
\[
 s_d(\{ (v_1,\ell_1) , (v_2,\ell_2) \}) = \delta_{\ell_1,\ell_2}
 s(\{v_1,v_2\})
\]
for all $\{v_1,v_2\} \in E$, $\ell_1,\ell_2 \in \{1,\ldots,d\}$, and
with $\delta_{\ell_1,\ell_2}$ being the Kronecker delta. Now take a
solution $v \in \real^V$ to the Dirichlet problem on $G$ with scalar
conductivity $s$ and let $e_i$ be the $i-$th canonical basis vector of
$\real^d$. Then up to a reordering of $V_d$, $v \kron e_i$ solves the
Dirichlet problem on $G$ with matrix valued conductivity $\sigma$ and
boundary data $v|_B \kron e_i$, $i=1,\ldots,d$. Here we used the
Kronecker product $\kron$ which we recall for convenience in
\eqref{eq:kron}. 
Let $v_j$ be the
solution to the Dirichlet problem on $G$ with conductivity $s$ such
that $v_j|_B = e_j$, $j=1,\ldots,|B|$.  Then we have
\[
 \range(W(s,s)) = \linspan\{ \vect((\nabla v_i) \outter (\nabla
 v_{i'})),~
 i,i'=1,\ldots,|B| \},
\]
where we used the outer product $\outter$ defined in \eqref{eq:outer}.
Since $\nabla(v \kron e_i) = (\nabla v) \kron e_i$ for any $v \in
\real^V$ we should have that 
\begin{equation}
(\nabla (v_i \kron e_j) \outter \nabla (v_{i'} \kron e_{j'})) =
 [ (\nabla v_i) \hadamard (\nabla v_{i'}) ] \kron (e_j e_{j'}^T),
\label{eq:prodrel}
\end{equation}
for any $i,i' = 1,\ldots,|B|$ and $j,j' = 1,\ldots,d$.  Now let us
 consider the subspace $M \subset \real^{d^2 |E|}$ spanned by
 all possible products
\eqref{eq:prodrel} in vector form, i.e.
\[
 M \equiv \linspan\left\{ \vect \left[ (\nabla (v_i \kron e_j)) \outter \nabla (v_{i'} \kron
 e_{j'})) \right],~i,i'=1,\ldots,|B|,~j,j'=1,\ldots,d \right\}.
\]
Since $\range(W(s,s)) = \real^E$ we deduce that $M = \real^{d^2|E|}$.
 Indeed, the $d^2$ subspaces associated with the pairs $(j,j') \in
 \{1,\ldots,d\}^2$ are mutually orthogonal and each has dimension $|E|$.
 The desired result follows because we have the inclusion $M \subset
 \range(W(\sigma,\sigma))$.
\end{proof}

\subsection{Matrix valued Schr\"odinger inverse problem}
Given a graph $G = (V,E)$ with boundary, the inverse problem we consider
here is to find the symmetric matrix valued Schr\"odinger potential $q \in
(\complex^{d\times d})^V$ from the Dirichlet to Neumann map
$\Lambda_{\sigma,q}$, where the conductivity $\sigma \in (\complex^{d
\times d})^E$ is symmetric with $\sigma' \succ 0$ and is assumed to be
known. This problem has the structure of the abstract inverse problem of
\cref{sec:structure}, as we see next.

\begin{itemize}
   \item The {\bf admissible set} is
    \[
     R = \left\{ q \in (\complex^{d \times d})^V ~|~ q =
     q^T~\text{and}~q_I' \succ -\lambda_{\min}((L_{\sigma'})_{II})
     \right\}.
    \]
   This is an open convex set in $(\complex^{d \times d})^V$ which can
   be identified to an open convex subset of $\complex^{d^2|V|}$. The
   {\bf forward map} is the map that to $q \in R$ associates the Dirichlet to
   Neumann map $\Lambda_{\sigma,q} \in \complex^{d|B| \times d|B|}$.
   This map is given by \eqref{eq:dtnmap} and is well defined for $q \in R$ because of \cref{thm:dir:pd}.

  \item The {\bf boundary/interior identity} is given by applying
   \cref{lem:interior} with $\sigma_i=\sigma$, $i=1,2$:
   \[
    g_2^T(\Lambda_{\sigma,q_1} - \Lambda_{\sigma,q_2})g_1 = b( S_{q_2}
  g_2 ,S_{q_1} g_1)^T \vect(q_1 - q_2).
   \]
  Here we define $S_q \in \complex^{d|V| \times d|B|}$ by its action
  on some $f \in \complex^{d|B|}$
  \[
   S_q f = u_I,
  \]
 where $u$ solves the Dirichlet problem \eqref{eq:dir} with $u_B=f$. The
 bilinear map $b : \complex^{d|V| \times d|V|} \to \complex^{|V|d^2}$
 is defined by $b(u,v) = u\outter v$, where the block-wise outer product $\outter$ is
 defined in \eqref{eq:outer} and we implicitly identify
 $(\complex^d)^V$ with $\complex^{d|V|}$ and $(\complex^{d \times
 d})^V$ with $\complex^{|V|d^2}$.

 \item {\bf Analyticity assumption.} From \cref{lem:spd}, the solution
  $u$ to the Dirichlet problem \eqref{eq:dir} with $u_B = f$ and $q \in
  R$ is determined by $u_I = - (L_{II} + \diag(q_I))^{-1} L_{IB}f$, where
  we omitted the subscript $\sigma$ from the graph Laplacian $L_\sigma$.
  Hence the entries of $S_q$ are analytic for $q \in R$.
\end{itemize}

\subsection{Rank deficient matrix valued conductivity inverse problem}
\label{sec:rkdefip}
Here we consider the inverse problem of recovering a conductivity
$\sigma \in (\complex^{d\times d})^E$ that is rank deficient from its
Dirichlet to Neumann map $\Lambda_{\sigma,0}$. Our
theory applies to a simpler problem, where we focus on finding the
eigenvalues $\lambda \in (\complex^r)^E$ of the conductivity $\sigma \in
(\complex^{d \times d})^E$, assuming that the eigenvectors $x \in
(\complex^{d \times r})^E$ are given, where used the  
notation of \cref{sec:sigpsd}. As in \cref{sec:sigpsd}, we have only
considered the case where the rank is $r$ for all edges. The more
general case of rank depending on the edges can also be dealt with, but
is not presented here for sake of simplicity. 

\begin{itemize}
   \item We take as {\bf admissible set}
    \[
     R = \left\{ \lambda \in (\complex^r)^E ~|~ \lambda' > 0 \right\}.
    \]
   Clearly $R$ is an open convex set in $(\complex^r)^E$, which can be
   identified to an open convex subset of $\complex^{r|E|}$. The forward
   map associates to $\lambda \in R$ the Dirichlet to Neumann map
  $\Lambda_{\sigma(\lambda),0}$ where $\sigma(\lambda)$ has eigenvectors $x$ and eigenvalues
  $\lambda$, i.e. $\sigma(\lambda)$ satisfies \eqref{eq:seigen}.
  \Cref{lem:dtn:psd} guarantees that this map is well defined for
  $\lambda \in R$.

  \item  Let
   $\lambda_1,\lambda_2 \in R$. By
   \cref{lem:interior} with $\sigma_i \equiv \sigma(\lambda_i)$ and
  $q_i=0$, $i=1,2$, we get the {\bf boundary/interior identity}
  \[
   g_2^T (\Lambda_{\sigma_1,0} - \Lambda_{\sigma_2,0}) g_1 =
  b(S_{\lambda_2}g_2,S_{\lambda_1}g_1)^T (\lambda_1 - \lambda_2).
  \]
 We define the matrix $S_\lambda \in \complex^{r|E|\times d|B|}$ by its
 action on some $f\in \complex^{d|B|}$,
 \[
  S_\lambda f = \diag(x)^T \nabla u,
 \]
 where $u$ solves the Dirichlet problem \cref{eq:dir} with boundary
 data $u_B = f$, conductivity $\sigma(\lambda)$ satisfying \cref{eq:seigen} and
 $q=0$. Recall that the Dirichlet problem solution is
 determined up to floppy modes. However from the floppy mode
 characterization in \cref{lem:floppy:char}, we see that the definition
 of $S_\lambda$ is independent of the choice of floppy mode. The
 bilinear map $b: \complex^{r|E|} \times \complex^{r|E|} \to
 \complex^{r|E|}$ is simply the Hadamard product, i.e. $b(u,v) = u
 \hadamard v$, and as before we identify $(\complex^r)^E$ with
 $\complex^{r|E|}$.

\item {\bf Analyticity assumption.} From the proof of \cref{lem:dtn:psd}
 (see \cref{sec:dtn:pinv}), a solution $u$ to the Dirichlet problem
  \cref{eq:dir} with boundary data $u_B = f$, conductivity
  $\sigma(\lambda)$ satisfying \cref{eq:seigen} and $q=0$, is determined
  by 
  \[
    u_I = - Q(Q^T (L_\sigma)_{II}Q)^{-1}Q^T L_{IB} f,
  \]
 where $Q$ is a real matrix such that $Q^T Q = I$ and $\range(Q) =
 \range((L_\sigma)_{II})$. Since $Q$ depends only on the graph and the
 (known a priori) eigenvectors $x$, the entries of $u_I$ are analytic for
 $\lambda \in R$. Hence the entries of $S_\lambda$ must also be analytic
 for $\lambda \in R$.
\end{itemize}

\section{Application to networks of springs, masses and dampers}
\label{sec:elnets}
\subsection{Spring networks}
\label{sec:springnets}
Consider a graph $G = (V,E)$ with boundary $B$ and let $p \in
(\real^d)^V$ be a function representing the equilibrium position of each node in
dimension $d=2$ or $3$. Each edge $e \in E$ represents a spring with positive spring
constant given by the function $k \in (0,\infty)^E$. 
Let $u \in (\real^d)^V$ denote the displacements of the nodes with
respect to the equilibrium position. The quantity $\nabla u \in
(\real^d)^E$ is the net spring displacement. By Hooke's law, the force
exerted by a spring 
is proportional to the net spring displacement. Here the proportionality
is given by a function $k \in
(0,\infty)^E$. For infinitesimally small displacements, the force
exerted by spring $\{i,j\} \in E$ is proportional to the projection of the net
displacement of spring $\{i,j\}$ on the direction $p(i)-p(j)$. In other
words, the forces are $\diag(\sigma) \nabla u$, where $\sigma \in
(\real^{d\times d})^E$ is the positive semidefinite conductivity
\begin{equation}
\label{eq:kcond}
 \sigma(\{i,j\}) = k(\{i,j\}) \frac{[p(i)-p(j)] [p(i)
 -p(j)]^T}{[p(i)-p(j)]^T [p(i) -p(j)]},~\text{for}~\{i,j\} \in E. 
\end{equation}
Now assume we displace the boundary nodes by an amount $g \in
(\real^d)^B$. If the interior nodes are left to move freely, the net
forces at the interior nodes should be zero, this condition is
equivalent to $(L_\sigma u)_I = 0$. Hence finding the displacements in a
spring network arising from (static) boundary displacements is the same
as solving the Dirichlet problem \eqref{eq:dir} with the particular
matrix valued conductivity \eqref{eq:kcond} and zero Schr\"odinger
potential. Using \cref{thm:dir:psd}, we see that the interior
displacements are uniquely determined by the boundary displacements (up
to floppy modes) and that the Dirichlet to Neumann map $\Lambda_\sigma$
is given by \cref{lem:dtn:psd}. In this particular case this map is
called {\em displacement to forces} map.

\subsection{Elastodynamic networks with damping}
\label{sec:elastodynamic}
We now consider the case where the displacements depend on time, i.e.
the function $u: V \times \real \to \real^d$ is defined such that
$u(i,t)$ is the displacement about the equilibrium position $p(i)$ of
node $i\in V$ at time $t \in \real$. We use the notation $\dot{u} =
du/dt$ and $\ddot{u} = d^2u/dt^2$ and we assume that all nodes have a
non-zero mass, which is given by the function $m \in (0,\infty)^V$.

\subsubsection{Viscous damping} 
\label{sec:damping}
We consider two kinds of viscous damping. The first is {\em spring
damping}, which is proportional to the net velocity of a spring
and is assumed to be in the same direction as the equilibrium position
of the springs, with proportionality constant given by a function $c_E
\in [0,\infty)^E$. This corresponds to having a damper in parallel with
each spring. The net forces associated with this damping are given by
$L_\mu \dot{u}$, where $\mu \in (\real^{d\times d})^E$ is defined by
\begin{equation}
  \label{eq:mucond}
  \mu(\{i,j\}) = c_E(\{i,j\}) \frac{[p(i)-p(j)] [p(i)
 -p(j)]^T}{[p(i)-p(j)]^T [p(i) -p(j)]},~\text{for}~\{i,j\} \in E.
\end{equation}
The second is {\em nodal damping}, meaning that each node is
inside a small cavity containing a viscous fluid and is thus subject to
a damping force proportional to the node velocity, with the
proportionality constant given by a function $c_V \in [0,\infty)^V$.
The forces associated with this kind of damping are
$\diag(q_{\text{damp}}) \dot{u}$ where $q_{\text{damp}} \in (\real^{d\times
d})^V$ is defined by $q_{\text{damp}}(i) = c_V(i) \identity$, for $i \in V$ and
$\identity$ being the $d\times d$ identity matrix. 

\subsubsection{Equations of motion in time domain}
Putting everything together and applying Newton's second law, we get the
equations of motion for an elastodynamic network:
\begin{equation}
 \diag(q_{\text{mass}}) \ddot{u} + (\diag(q_{\text{damp}}) + L_\mu)
 \dot{u} + L_\sigma u = f,
 \label{eq:motion1}
\end{equation}
where $q_{\text{mass}} \in (\real^{d\times d})^V$ is defined by
$q_{\text{mass}}(i) = m(i) \identity$ for $i \in V$. The function $f : V
\times \real \to \real^d$ is a function representing any external
forces, i.e. $f(i,t)$ is the external force exerted on node $i \in V$ at
time $t$. This second order system of ordinary differential equations
can be written as
\begin{equation}
 M \ddot{u} + C \dot{u} + K u = f,
 \label{eq:motion2}
\end{equation}
where $M = \diag(q_{\text{mass}})$ is the {\em mass matrix}, $C =
\diag(q_{\text{damp}})+ L_\mu$ is the {\em damping matrix} and $L_\sigma$ is
the {\em stiffness matrix}.

\subsubsection{Frequency domain formulation and the Dirichlet problem}
\label{sec:fdfdp}
For a time harmonic displacement $u(i,t) = \exp[\jmath\omega t] \hat
u(i,\omega)$, the equations of motion \eqref{eq:motion2} become
\begin{equation}
 (-\omega^2 M + \jmath\omega C + K) \hat u = \hat f.
 \label{eq:motion3}
\end{equation}
Now consider the problem of finding the (frequency domain) displacements
$\hat u_I$ at the interior nodes knowing the displacements $\hat u_B$ at
the boundary nodes and that there are no external forces at the interior
nodes (i.e.  $\hat f_I = 0$).  We immediately see that we have another
instance of the Dirichlet problem \eqref{eq:dir} with complex
conductivity $\sigma + \jmath\omega \mu$ and complex Schr\"odinger
potential $-\omega^2 q_{\text{mass}} + \jmath\omega q_{\text{damp}}$.
Unfortunately we cannot apply \cref{thm:dir:pd} directly because we do
not have $\sigma \succ 0$ or $-\omega^2 q_{\text{mass}} \succ 0$. To
remedy this, we assume there is always a small amount of damping at the
nodes i.e. $c_V \in (0,\infty)^V$, in a way reminiscent of the limiting
absorption principle for the Helmholtz equation. We rewrite the
equations of motion \eqref{eq:motion3} as follows
\begin{equation}
  (\jmath\omega M + C + (\jmath\omega)^{-1} K)(\jmath \omega \hat u) = \hat f.
\end{equation}
Again if the forces at the interior nodes are equilibrated, this is an
instance of the Dirichlet problem \eqref{eq:dir} with complex
conductivity $\mu + (\jmath\omega)^{-1} \sigma$ and complex
Schr\"odinger potential $q_{\text{damp}} + \jmath\omega
q_{\text{mass}}$.  A positive damping at the nodes guarantees $
q_{\text{damp}} \succ 0$. Thus the Dirichlet problem admits a unique
solution by \cref{thm:dir:pd}. Indeed the condition $(L_{\mu})_{II}
\succ -\lambda_{\min}(q_{\text{damp}})$ always holds in this case
because $(L_{\mu})_{II} \succeq 0$. Hence  the Dirichlet to Neumann map
$\Lambda_{\mu + (\jmath\omega)^{-1} \sigma , q_{\text{damp}} +
\jmath\omega q_{\text{mass}}}$ is well defined by \eqref{eq:dtnmap} and
so is the Dirichlet to Neumann map for the original problem:
$\Lambda_{\sigma + \jmath\omega\mu, -\omega^2 q_{\text{mass}} +
\jmath\omega q_{\text{damp}}}$, as can be seen from a homogeneity
argument. Since the latter map associates the frequency domain
displacements to frequency domain forces, we also call it {\em
displacement to forces map}.

\subsection{Spring constant inverse problem: static case}
Let us consider the inverse problem of finding the spring constants $k
\in \real^E$ from
the static displacement to forces map $\Lambda_{\sigma(k),0}$ of a network of springs. 
We assume
the equilibrium positions $p \in (\real^d)^V$ of the nodes are known.
Uniqueness for this inverse problem can be established using the result
in \cref{sec:rkdefip} for rank deficient matrix valued conductivities,
which we adapt here to this particular problem. Since
we are not aware of a physically relevant interpretation of complex
valued spring constants in the static case, we take spring constants in the {\bf
admissible set}
\[
 R = (0,\infty)^E.
\]
The forward map associates to $k \in R$, the displacement to forces map
$\Lambda_{\sigma(k),0}$. The conductivity $\sigma(k)$ is defined in
\eqref{eq:kcond}. For an edge $\{i,j\}$, the spring constant
$k(\{i,j\})$ is the only non-zero eigenvalue of the conductivity
$\sigma(\{i,j\})$. To write the boundary/interior identity for this
problem we introduce the function $x\in (\real^d)^E$ that to an edge
$\{i,j\} \in E$ associates the corresponding normalized eigenvector,
i.e.
\[
x(\{i,j\}) = \frac{p(i)-p(j)}{\|p(i)-p(j)\|}.
\]
The {\bf boundary/interior identity} is then
\[
 g_2^T ( \Lambda_{\sigma_1,0} - \Lambda_{\sigma_2,0}) g_1 = ([ S_{k_1}
 g_1] \hadamard [S_{k_2} g_2 ])^T (k_1
 - k_2),
\]
where $\sigma_i \equiv \sigma(k_i)$, $i=1,2$ and $g_1,g_2$ are vectors in
$\real^{d|B|}$. The matrix $S_{k_i} \in \real^{|E| \times d|B|}$ is
defined such that the vector $S_{k_i} g_i$ contains the components of
$\nabla u_i$ along the spring directions, i.e.
\begin{equation}
\label{eq:ski}
 (S_{k_i} g_i)(e) = x(e)^T (\nabla u_i)(e),~e\in E,
\end{equation}
where $u_i$ is the displacement arising from displacing the boundary
nodes by $g_i$, $i=1,2$. Concretely, this problem fits the mold of
\cref{sec:structure}. Thus from \cref{sec:uniqueness:ae}, if the
linearization of the inverse problem for the spring constants in a
spring network is injective for particular spring constants, then it
must also be injective for almost all other spring constants.

\subsection{Spring constant inverse problem assuming masses are known}
Here we consider the problem where the operating frequency $\omega$, the
equilibrium position of the nodes $p \in (\real^d)^V$, the masses $m \in
(0,\infty)^V$ and mass dampers $c_V \in (0,\infty)^V$ are all known, but
we want to recover the spring constants $k \in (0,\infty)^E$ and spring
dampers $c_E \in (0,\infty)^E$ from the displacement to forces map at
the frequency $\omega$.
\begin{itemize}
   \item The
   {\bf admissible set} is
   \[
    R =  \{  \rho \in \complex^E~|~  \rho'>0,
    \sign(\omega)  \rho''>0\},
   \]
   where we grouped for convenience the spring constants and spring
   dampers into a single complex valued edge function $\rho$. To be more
   precise, if $\rho \in R$ we have $\rho' = k$ and $\rho'' = \omega
   c_E$.  The {\bf forward map} associates to $\rho \in R$ the
   displacement to forces map $\Lambda_{\sigma(\rho),q}$, where
   $\sigma(\rho)$ is defined as in \eqref{eq:kcond} with $k\equiv \rho$
   and $q \equiv -\omega^2q_{\text{mass}} + \jmath\omega
   q_{\text{damp}}$. The forward map is well defined and given by
   \eqref{eq:dtn} for all $\rho \in R$ because we assumed damping at the
   nodes, see \cref{sec:fdfdp}.

   \item The {\bf boundary/interior identity} is
   \[
    g_2^T(\Lambda_{\sigma_1,q} - 
          \Lambda_{\sigma_2,q}) g_1
	  = 
	  ( [S_{\rho_1} g_1] \hadamard [S_{\rho_2} g_2 ])^T (\rho_1 -
	  \rho_2),
   \]
   where $\sigma_i \equiv \sigma(\rho_i)$,  $g_i \in \complex^{d|B|}$ and
   the matrix $S_{k_i}$ is defined as in \cref{eq:ski} for $i=1,2$.

   \item {\bf Analyticity assumption.} We can use \cref{lem:spd} to
   guarantee that the solution to the Dirichlet problem with boundary
   displacements $u_B = f$ is given by $u_I = -(L_{II} +
   \diag(q_I))^{-1}L_{IB}f$, where we omitted the subscript $\sigma(\rho)$ from the
   Laplacian $L_{\sigma(\rho)}$. This implies the entries of $S_{\rho}$ are
   analytic for $\rho \in R$.
\end{itemize}
Thus the problem of finding the spring constants when the masses are
known fits the mold of \cref{sec:structure}. 
The above argument can be adapted to the case where there are no spring
dampers, i.e. $c_E = 0$. In this case the admissible set would be
$R = (0,\infty)^E$. However the problem no longer satisfies the
assumptions of \cref{sec:structure} if we 
do not know if spring dampers are present, because with $c_E \in
[0,\infty)^E$ the admissible set would not be open.

\subsection{Mass inverse problem assuming spring constants are known}
Here we assume that the operating frequency $\omega$, the equilibrium
position of the nodes $p \in (\real^d)^E$, the spring constants $k \in
(0,\infty)^E$ and the spring dampers $c_E \in [0,\infty)^E$ are known.
The inverse problem is to find the masses $m \in (0,\infty)^V$ and nodal
dampers $c_V \in (0,\infty)^V$ from the displacement to forces map at
the frequency $\omega$. As we see next, this problem also satisfies the
assumptions of \cref{sec:structure}.
\begin{itemize}
   \item The {\bf admissible set} is 
   \[
    R = \{ \rho \in \complex^V ~|~ \rho' < 0, \sign(\omega)\rho'' > 0\},
   \]
   where we grouped for convenience the masses and nodal dampers into a
   single complex valued nodal function $\rho$. If $\rho \in
   R$, then $\rho' = -\omega^2 m$  and
   $\rho'' = \omega c_V$.  The {\bf forward map} associates to $\rho \in R$
   the displacement to forces map $\Lambda_{\sigma,q(\rho)}$, where
   $\sigma$ is defined as in \eqref{eq:kcond} with $k \equiv k + \jmath
   \omega c_E$, and $q(\rho) \in (\complex^{d\times d})^V$ is defined by
   $q(\rho)(i) = \rho(i) \identity$ for all vertices $i \in V$.  The
   displacement to forces map is well defined and given by
   \eqref{eq:dtn} for all $\rho \in R$ because we assumed damping at the
   nodes, see \cref{sec:fdfdp}.

   \item The {\bf boundary/interior identity} is
   \[
    g_2^T ( \Lambda_{\sigma,q(\rho_1)} - \Lambda_{\sigma,q(\rho_2)}) g_1
    =
    ( u_1 \hadamard u_2 )^T\vect(q(\rho_1-\rho_2)),
   \]
   where $u_i$ solves the Dirichlet problem \eqref{eq:dir} with
   conductivity $\sigma + \jmath \omega \mu$ and Schr\"odinger potential
   $q(\rho_i)$, $i=1,2$.

   \item {\bf Analyticity} follows from the solution $u$ to the
   Dirichlet problem being well defined from all Schr\"odinger
   potentials of the form $q(\rho)$ (see the discussion in
   \cref{sec:fdfdp}).
\end{itemize}

\section{Summary and Perspectives}
\label{sec:summary}
We have presented several inverse problems on graphs that share the
common structure of \cref{sec:structure}. In these inverse problems, the
unknowns are matrices (or their eigenvalues) defined on the edges or
nodes of a graph (\cref{sec:matrix:problems}). By giving sufficient
conditions under which the Dirichlet problem on a graph with matrix
valued weights admits a unique solution we can deduce a set of
parameters on which the forward map is analytic. In cases where the weights
are rank deficient, the solution is not unique but can be determined up
to ``floppy modes'' that depend only on the nullspaces of the
weights (\cref{sec:matrix:problems,sec:diruni}). Thus the forward map
can still be shown to be analytic in this case.  Analyticity of the
forward map and its Jacobian have practical consequences that are given
in \cref{sec:uniqueness:ae,sec:applications}. Particular examples of
inverse problems on graphs are given in \cref{sec:matrixval:ex}, with a focus
inverse problems in elastodynamic networks (\cref{sec:elnets}) at a
single frequency. Multi-frequency or time domain problems are left for
future studies.

There remains many open questions. For example, it is not clear how to
find a graph on which uniqueness a.e. holds for a given problem. This
was done in \cite{Boyer:2015:SDC} by trying many random graphs drawn
from the Erd\H{o}s-R\'enyi model \cite{Erdos:1959:ORG}. A similar
approach could be taken here.  It is also not clear whether direct
solution methods such as the layer peeling algorithm in
\cite{Curtis:1994:FCC} exist for these matrix inverse problems on
networks. Finally, the theoretical results we present here rely on
analytic continuation, which is a notoriously unstable procedure. 

\section*{Acknowledgments}
This work was supported by the National Science Foundation grant
DMS-1411577. Most of these results were derived in the
Spring 2016 introduction to research class called ``Network Inverse
Problems'', supported by the same grant, at the University of Utah. FGV
also thanks support from the National Science Foundation grant
DMS-1439786, while FGV was in residence at the Institute for
Computational and Experimental Research in Mathematics in Providence,
RI, during the Fall 2017 semester. FGV thanks Laboratoire Jean Kuntzmann for
hosting him during the 2017-2018 academic year. Finally we thank
Vladimir Druskin for his comments on an earlier version of this
manuscript.
\appendix

\section{Facts about analytic functions of several complex variables}
\label{app:faf}
A function $f : \complex^n \to \complex$ is analytic on some open set $R
\subset \complex^n$ if for any $z_0 \in R$, the function $f(z)$ can be
expressed as a convergent power series, i.e. we can find complex coefficients
$c_\alpha$ for which the series
\[
 f(z) = \sum_{\alpha \in \nat^n} c_\alpha (z-z_0)^\alpha,
\]
converges for all $z \in R$. Here we used the notation $z^\alpha =
z_1^{\alpha_1} z_2^{\alpha_2} \cdots z_n^{\alpha_n}$, for a multi-index
$\alpha \in \nat^n$. Rational functions of the form $P(z) / Q(z)$,
for $P(z)$ and $Q(z)$ polynomials, are analytic on any connected open
set where $Q(z) \neq 0$. Moreover, the product and the sum of two
analytic functions is also analytic. The uniqueness lemma below is a
consequence of analytic continuation, i.e. if $f(z)$ is analytic for
$z\in R$ and we can find $z_0$ such that $f(z_0) \neq 0$, then the {\em
zero set} of $f$
\[
 Z \equiv \{ z \in R ~|~ f(z) = 0 \},
\]
must be a set of measure zero with respect to the Lebesgue measure on
$R$ (see e.g. \cite{Gunning:1965:AFS}).

\bibliographystyle{siamplain}
\bibliography{schroe}

\end{document}